\documentclass[preprint,12pt]{elsarticle}
\usepackage{amsmath,amsthm}
\usepackage{graphicx}
\usepackage{amsfonts}
\usepackage{amssymb}
\usepackage{subfigure}
\usepackage{url}
\usepackage{enumerate}

\newtheorem{theorem}{Theorem}[section]

\newtheorem{definition}[theorem]{Definition}

\newtheorem{lemma}[theorem]{Lemma}

\newtheorem{proposition}[theorem]{Proposition}
\newtheorem{remark}[theorem]{Remark}

\numberwithin{equation}{section}

\journal{Physica D}

\begin{document}

\begin{frontmatter}
\title{Shape selection in non-Euclidean plates}
\author[az]{John A. Gemmer\corref{cor1}}
\ead{jgemmer@math.arizona.edu}
\author[az,bz]{Shankar C. Venkataramani}
\ead{shankar@math.arizona.edu}

\address[az]{Program in Applied Mathematics,University of Arizona, Tucson, AZ 85721, U.S.A.}
\address[bz]{Department of Mathematics, University of Arizona, Tucson, AZ 85721, U.S.A.}
\cortext[cor1]{Principal corresponding author}


\begin{abstract} We investigate isometric immersions of disks with constant negative curvature into $\mathbb{R}^3$, and the minimizers for the bending energy, {\em i.e.} the $L^2$ norm of the principal curvatures over the class of $W^{2,2}$ isometric immersions. We show the existence of smooth immersions of arbitrarily large geodesic balls in $\mathbb{H}^2$ into $\mathbb{R}^3$. In elucidating the connection between these immersions and the non-existence/singularity results of Hilbert and Amsler, we obtain a lower bound for the $L^\infty$ norm of the principal curvatures for such smooth isometric immersions. We also construct piecewise smooth isometric immersions that have a periodic profile, are globally $W^{2,2}$, and numerically have lower bending energy than their smooth counterparts. The number
of periods in these configurations is set by the condition that the principal
curvatures of the surface remain finite and grow approximately exponentially with the radius of the disc.
We discuss the implications of our results on recent experiments on the mechanics of non-Euclidean plates.
\end{abstract}

\end{frontmatter}

\section{Introduction}

The differential growth of thin elastic sheets can generate highly non-trivial configurations such as the multiple generation of waves along the edge of a torn elastic sheet \cite{Sharon2002} and the wrinkled patterns on leaves \cite{UtpalNath02282003}. Recently, environmentally responsive gel disks of uniform thickness have been created in the laboratory that mimic this type of growth by differentially shrinking in the radial direction when thermally activated in a hot water bath \cite{Shaping-of-elastic-sheets-by-prescription-of-non-euclidean-metrics}. The rate of shrinkage can be controlled, allowing for the creation of regions in the disk that shrink at different rates upon activation. Rippling occurs in regions that shrink slower than the center of the disk while regions that shrink faster obtain a spherical dome like shape. Geometrically, these regions have positive and negative Gaussian curvature respectively and by tuning the composition of the gel, different curvatures can be selected. This type of growth can be used to generate a multitude of different shapes of varying geometries, but a mathematical framework that can accurately predict the equilibrium shapes is lacking.

Classically, thin elastic plates have been modeled by the F\"oppl -- von K\`arm\`an (FvK) equations that reduce the full three dimensional equations of elasticity to a coupled system of partial differential equations defined on the mid-surface of the plate \cite{Landau, foppl}. These equations are asymptotically valid in the small thickness regime if the thickness scales in a particular manner with external loadings \cite{CIARLET}, and different scalings lead to a hierarchy of different plate models that includes FvK \cite{MUELLER}. The FvK equations can be obtained as the variation of an energy functional. This energy is the sum of stretching and bending terms such that the ratio of bending to stretching energy scales like the thickness squared \cite{Landau}. Consequently, thin plates energetically prefer to bend out of the plane rather than stretch -- a phenomenon easily observed in stretched sheets \cite{Cerda-2003} -- and this could be the mechanism used by growing sheets to buckle out of the plane. 

Non-uniform growth is difficult to model using the FvK equations since it commonly generates structures with residual stresses \cite{Goriely-2005} and thus it is unclear what is an appropriate reference configuration from which strains can be measured. One technique for defining a reference configuration is to specify a two dimensional ``target metric'' $\mathbf{g}$  on the mid-surface of the sheet from which strains are measured  \cite{Theory-of-Edges-of-Leaves, Geometry-And-Elasticity-of-srips-and-flowers, B.Audoly2002, PhysRevLett.91.086105}. The metric $\mathbf{g}$ is a geometric quantity that locally determines all of the intrinsic properties of an abstract surface such as its Gaussian curvature and geodesics \cite{gray}. Thus, we are naturally led to consider the problem of finding isometric immersions of $\mathbf{g}$ into $\mathbb{R}^3$  \cite{2001JPhA...3411069N}. But, the problem is that in general isometric immersions of $\mathbf{g}$ may not exist or be unique and this model does not incorporate important three dimensional effects such as a mechanism for penalizing highly bent configurations. 

Instead of modeling non-uniform growth as isometric immersions, a ``non- Euclidean" elastic energy functional has been recently introduced that incorporates $\mathbf{g}$ and reduces to the FvK functional when $\mathbf{g}$ is the identity matrix \cite{ERAN}. This functional has the same energy scaling as FvK with the  stretching energy measuring $L^2$ deviations of the immersion from $\mathbf{g}$.  As with the FvK functional, if the thickness is sufficiently small the stretching energy dominates and we expect the elastic sheet to bend out of the plane to match the metric. In particular, if $W^{2,2}$ isometric immersions exist then in the limit of vanishing thickness minimizers of the non-Euclidean FvK functional converge to a minimizer of the bending energy over the class of $W^{2,2}$ isometric immersions of $\mathbf{g}$ \cite{Gamma-limit}. Thus in this model the particular isometric immersion is selected by the bending energy.

An alternative is to decompose the deformation of the sheet into a growth tensor and an elastic response, and define the elastic energy in these terms. Lewicka, Mahadevan and Pakzad \cite{gamma_convergence_smallslopes} have shown that the two approaches are essentially equivalent in the limit that the thickness of the sheet vanishes.

In this paper we study isometric immersions of $\mathbf{g}$ in the disk geometry when $\mathbf{g}$ specifies a constant negative Gaussian curvature $K_0$. Experimentally, disks of constant negative curvature obtain an equilibrium configuration with a periodic profile of one wavelength with the number of nodes increasing with decreasing thickness \cite{an-experimental-study-of-shape-transitions-and-energy-scaling-in-thin-non-euclidean-plates}. Consequently, we will look for low bending energy test functions that match the target metric and have a periodic profile.

Finding such test functions corresponds to the classical problem of finding isometric immersions of the hyperbolic plane $\mathbb{H}^2$ into $\mathbb{R}^3$. Hilbert \cite{Hilbert} proved that there are no real analytic isometric immersions of the entire hyperbolic plane into $\mathbb{R}^3$,  but this does not exclude the possibility of the existence of isometric immersions with weaker smoothness criteria. In fact, the Nash-Kuiper theorem states that through the technique of convex integration there exists $C^1$ isometric immersions of $\mathbb{H}^2$ \cite{Nash-1954, Kuiper-1955}. But, such an immersion would be too rough since the bending energy requires the surface to be at least weakly second differentiable. Furthermore, by Efimov's theorem, we know that there are no $C^2$ isometric immersions of $\mathbb{H}^2$ \cite{Efimov-1964}. The admissible surfaces with finite elastic FvK energy are $W^{2,2}$ and this motivates the search for $W^{2,2}$ isometric immersions which lie between $C^1$ and $C^2$ in smoothness. 

 In general finding an isometric immersion of an arbitrary metric is a non-trivial problem. But, by Minding's theorem \cite{Minding}, which states that any two abstract surfaces having the same constant Gaussian curvature are locally isometric, we can cut out subsets of surfaces of constant negative curvature which will be isometric to pieces of $\mathbb{H}^2$. In particular, if $S$ is a surface of constant negative curvature and $p\in S$ then $U=\{q\in S:d(p,q)<R\}$ is isometric to a disk of radius $R$ in $\mathbb{H}^2$, provided $U$ is smooth. Furthermore, it was proven by Amsler that the singularities where $S$ fails to be smooth form curves on $S$ \cite{Amsler}. These curves form a boundary beyond which no disks locally isometric to $\mathbb{H}^2$ can be cut out of the surface.

This paper is organized as follows. In Sec.~\ref{sec:model} we present the mathematical model for thin elastic sheets with a specified metric $\mathbf{g}$ in the form of a variational problem with a FvK type functional. We show that for isometric immersions the functional reduces to the Willmore functional and thus in the vanishing thickness limit if an isometric immersion exists the minimizing configurations should converge to a minimizer of the Willmore energy over the class of $W^{2,2}$ isometric immersions \cite{Gamma-limit}.

In Sec.~\ref{sec:hyperbolic} we review some of the theory of hyperbolic surfaces that will allow us to construct isometric immersions in $\mathbb{R}^3$. In particular, we make use of Chebyshev nets (C-nets) which are local parameterizations of a surface whose coordinate curves form quadrilaterals on the surface with opposite sides having equal length \cite{gray}. Originally, Chebyshev used C-nets to describe the arrangement of fibers in clothing and he showed that a sphere must be covered by at least two cuts of cloth \cite{chebyshef-1878} (see \cite{McLachlan-1994} for a modern translation). As evidenced by the sphere, C-nets in general are only local parameterizations of a surface and a global C-net exists if the integral of the magnitude of the Gaussian curvature over a quadrilateral is less than $2\pi$ \cite{Hazzidikas-1880}. C-nets are useful parameterizations in that they lead to a natural way to discretize a surface \cite{Bobenko-1999} and projections of C-nets onto paper retain the three dimensional properties of the surface \cite{Koenderink-1998}. Remarkably, for surfaces with constant negative curvature, every sufficiently smooth isometric immersion generates a C-net on the surface obtained by 
taking the asymptotic curves as the coordinate curves and conversely every C-net on a surface of constant negative curvature generates an immersion whose asymptotic curves are the coordinate lines \cite{gray}. The extrinsic and intrinsic geometry of the particular immersion is entirely determined by the angle $\phi$ between the asymptotic curves and a parameterization is a C-net if and only if $\phi$ is a solution of the sine-Gordon equation \cite{gray}. Thus, solutions of the sine-Gordon equation generate surfaces of negative curvature and the corresponding immersion will fail to be smooth when $\phi$ is an integer multiple of $\pi$.

In Sec.~\ref{sec:pseudosphere} we show that isometric immersions of geodesic discs in the hyperbolic plane can by constructed by cutting out subsets of a classical hyperbolic surface, the pseudosphere. This is contrary to the claims in \cite{0295-5075-86-3-34003} where the experimentally observed buckling and refinement in wavelength with decreasing thickness is attributed to the non-existence of smooth isometric immersions. These configurations have very large bending energy even away from the singular curve and do not have the rotational $n$-fold symmetry observed in experiments \cite{an-experimental-study-of-shape-transitions-and-energy-scaling-in-thin-non-euclidean-plates}. 

A natural question is the connection between this existence result and the non-existence result of Holmgren \cite{Holmgren} stating that given a ``local" immersion of an open set in $\mathbb{H}^2$ into $\mathbb{R}^3$, there is a maximal domain of finite diameter on which this immersion can be extended to a smooth global immersion with one of the principal curvatures diverging as we approach the boundary. In Sec.~\ref{sec:minimax}, we obtain a quantitative version of this result, by showing numerically that for isometric immersions of a Hyperbolic disk of radius $R$ the (nondimensional) maximum principal curvature of the immersion is bounded from below by a bound which grows exponentially in the (nondimensional) radius. We also identify the surfaces which realize this bound, and they turn out to be disks cut from Hyperboloids of revolution with constant negative curvature.

In Sec.~\ref{sec:small_slopes}, we use the small slopes approximation to find approximate isometric immersions with a periodic profile. The small-slopes approximation has been widely used to model torn sheets and growing tissues \cite{PhysRevLett.91.086105, Dervaux(2008), HaiyiLiang12292009}. In this setting the problem of finding immersions of $\mathbf{g}$ reduces to solving a hyperbolic Monge-Ampere equation. The simplest solutions are a family of quadratic surfaces with two asymptotic lines intersecting at the origin. We then periodically extend the surface bounded between the two lines to construct periodic shapes that have the same scaling of amplitude versus period observed experimentally \cite{an-experimental-study-of-shape-transitions-and-energy-scaling-in-thin-non-euclidean-plates}. The two wave configuration is the global minimum of the bending energy over the class of isometric immersions and we prove that with decreasing thickness the configurations should converge to this shape. 

  In Sec.~\ref{sec:amsler} we extend the surfaces constructed from solutions to the Monge-Ampere equation by studying hyperbolic surfaces with two asymptotic lines. These surfaces are called Amsler surfaces, named after Henri Amsler who studied them in his paper on the nature of singularities of hyperbolic surfaces \cite{Amsler, Bobenko}. Amsler surfaces are generated by a similarity transformation of the sine-Gordon equation which transforms it to a Painlev\'{e} III equation in trigonometric form \cite{Amsler, Bobenko}. From these surfaces, we can create periodic shapes as we did with the small slopes approximations that are energetically preferred over corresponding disks cut from the pseudosphere and hyperboloids. The bending energy of these surfaces is concentrated in small regions near the singular curves and the lines we are taking the periodic extension. Furthermore for each $n\geq 2$, there is a radius $R_n\sim \log(n)$ such that the $n$-periodic Amsler surfaces only exist for a radius $0<R<R_n$. This gives a natural geometric mechanism for the refinement of the wavelength of the buckling pattern with increasing radius of the disk.  Finally, we present a concluding discussion in Sec.~\ref{sec:discuss}.

%%%%%%%%%%%%%%%%%%%%%%%%%%%%%%%%%%%%%%%%%%%%%%%%%%%%%%%%%%%%%%%%%%%%%%%%%%%%%%
%
%			Model
%
%%%%%%%%%%%%%%%%%%%%%%%%%%%%%%%%%%%%%%%%%%%%%%%%%%%%%%%%%%%%%%%%%%%%%%%%%%%%%%%
\section{Model} \label{sec:model}

We model the stress free configuration of a thin elastic disk of thickness $t$ and radius $R$ that has undergone differential growth as an abstract Riemannian manifold $D$ with metric $\mathbf{g}$. More specifically, if we choose a coordinate system for $D$ so that $0$ lies at the center of the disk then $D=\{p:d(p,0)<R\}$, where $d$ is the distance function induced by $\mathbf{g}$. A configuration of the sheet in three dimensions is then given by a mapping $\mathbf{x}:D\rightarrow
\mathbb{R}^3$, where $U=\mathbf{x}(D)$, the image of the mapping, is taken to be the mid-surface of a sheet. The equilibrium configurations of the sheet are extrema for the elastic energy functional:
\begin{equation} \label{energy}
\mathcal{E}[\mathbf{x}]=\int_{D}\|D\mathbf{x}^T\cdot D\mathbf{x}-\mathbf{g}\|^2\,dA+t^2\int_{D}\|4H^2-K\|^2\,dA,
\end{equation}
where $H$ is the mean curvature of the surface $U$, $K$ is the Gaussian curvature, $D\mathbf{x}$ is the Jacobian matrix of partial derivatives of $\mathbf{x}$, and $dA$ is the area element on $D$ \cite{Shaping-of-elastic-sheets-by-prescription-of-non-euclidean-metrics}. This functional is an asymptotic expression for a full three dimensional elastic energy functional per unit thickness \cite{Gamma-limit}. The first integral measures the stretching energy and vanishes when $\mathbf{x}$ is an isometric immersion while the second integral measures bending energy and vanishes when the image of $\mathbf{x}$ is a plane.

This functional is a FvK type since it is the sum of stretching and bending energies such that ratio of bending to stretching energy scales like $t^2$ \cite{CIARLET}.  If $\mathbf{g}$ is not Euclidean, i.e. if the Gaussian curvature corresponding to $\mathbf{g}$ is non-zero, then both of these terms cannot simultaneously vanish in three dimensions and consequently the equilibrium configuration will not be able to reduce all of its elastic energy even in the absence of external forces. 

In this paper we will focus on (\ref{energy}) when $\mathbf{g}$ generates a constant negative Gaussian curvature $K_0 < 0$. $K_0$ is the {\em target Gaussian curvature}, and we can always nondimensionalize such that $K_0 = -1$, {\em i.e.}, we chose the units for length as $(-K_0)^{-1/2}$. If the surface is (initially) parameterized by $u,v$ then $u^{\prime}=\sqrt{-K_0}u$, $v^{\prime}=\sqrt{-K_0}v$, are the dimensionless parameterization variables and $R^{\prime}=\sqrt{-K_0}R$ is the dimensionless radius. Therefore, if we can find a configuration $\mathbf{x}$ such that $D\mathbf{x}^TD\mathbf{x}-\mathbf{g}=0$ then, the Gaussian curvature of the immersion $K=K_0=-1$ and (\ref{energy}) reduces to the following equivalent elastic energy functional
\begin{equation}\label{bending-energy-general}
\mathcal{B}[\mathbf{x}]=\int_D(k_1^2+k_2^2)\,dA,
\end{equation}
where $k_1$ and $k_2$ are the principal curvatures of $U$. Since $k_1k_2 = K = -1$, this energy is equivalent to the Willmore energy \cite{willmore}. We will refer to $\mathcal{B}$ as the (normalized) {\em bending energy} of the configuration $\mathbf{x}$, although, strictly speaking, the bending energy is the quantity $t^2 \mathcal{B}$ and not the quantity $\mathcal{B}$ itself.

Since we are integrating over a geodesic circle it is convenient to work in geodesic polar coordinates, which are the natural analogs of the radial and polar angle coordinates on a Riemannian manifold \cite{gray}. In these coordinates $U$ is parameterized by the coordinates $(r,\Psi)$ which map to a point $q\in U$ satisfying $d(q,p)=r$ and $q$ lies on a geodesic whose initial velocity vector makes an angle $\Psi$ with respect to a basis of the tangent plane at $p$. In this coordinate system the metric $\mathbf{g}$ has components 
\begin{equation}\label{HyperbolicMetric}
g_{11}=1, \text{ } g_{12}=g_{21}=0, \text{ and } g_{22}=\sinh(r),
\end{equation} 
and equation (\ref{bending-energy-general}) takes the form:
\begin{equation}\label{bending-energy-polar}
\mathcal{B}[\mathbf{x}]=\int_0^{2\pi}\int_0^R \sinh(r)(k_1^2+k_2^2)\,dr\,d\Psi.
\end{equation}

Lewicka and Pakzad \cite{Gamma-limit} have shown that, for immersions $\mathbf{x}$ in the space  $W^{2,2}(D,\mathbb{R}^3)$,
\begin{equation} \label{Gamma-limit-energy}
\Gamma-\lim_{t \rightarrow 0} \,\,t^{-2} \mathcal{E}[\mathbf{x}] = \begin{cases}  \mathcal{B}[\mathbf{x}], & D\mathbf{x}^T\cdot D\mathbf{x} = \mathbf{g} \\
+\infty, & \mbox{otherwise}
\end{cases}
\end{equation}
where $D\mathbf{x}^T\cdot D\mathbf{x} = \mathbf{g}$ implies that $\mathbf{x}$ is an isometric immersion. Therefore, in the limit as the thickness vanishes, equation (\ref{bending-energy-polar}) is the appropriate functional with the set of admissible functions $\mathcal{A}$ being sufficiently regular isometric immersions of the disk of radius $R$ in $\mathbb{H}^2$. In particular, we need $\mathcal{B}[\mathbf{x}] < \infty$, so the principal curvatures are $L^2$ functions on the disk $D$, and we need to consider $W^{2,2}$ immersions $\mathbf{x}:D \rightarrow \mathbb{R}^3$.
 
%%%%%%%%%%%%%%%%%%%%%%%%%%%%%%%%%%%%%%%%%%%%%%%%%%%%%%%%%%%
%
%					Hyperbolic Surfaces
%
%%%%%%%%%%%%%%%%%%%%%%%%%%%%%%%%%%%%%%%%%%%%%%%%%%%%%%%%%
\section{Theory of Hyperbolic Surfaces} \label{sec:hyperbolic}

In this section we show how isometric immersions of $D$ can be cut out of a surface $S$ with constant negative Gaussian curvature. Let $\mathbf{x}:D\rightarrow S$ be a parametrization of $S$ and for $p\in S$ consider the set $U=\{q\in S,d(p,q)<R\}$. By Minding's theorem \cite{Minding} as long as $U$ is smooth it can form an isometric immersion of $D$. To construct $U$ we need to calculate the geodesics on $S$ and also to determine $\mathcal{B}[\mathbf{x}]$ we need to calculate the principal curvatures on $S$.

To calculate these quantities recall the following facts and definitions from the theory of hyperbolic surfaces (see pgs. 683-692 in \cite{gray}). An asymptotic curve is a curve in $S$ whose normal curvature vanishes everywhere. Moreover, since the principal curvatures $k_1$ and $k_2$ of a hyperbolic surface have opposite signs it follows that through each point $p\in U$ there are exactly two asymptotic curves that intersect at $p$. Consequently, we can locally parameterize any hyperbolic surface by its asymptotic lines. Such a parameterization is a C-net and has a metric $\mathbf{g}^{\prime}$ with the following components
\begin{equation} \label{Tcebyshef-Metric}
g^{\prime}_{11}=g^{\prime}_{22}=1 \text{ and } g^{\prime}_{21}=g^{\prime}_{21}=\cos(\phi),
\end{equation}
where $\phi$ is the angle between the asymptotic curves at each point in $S$. Therefore, all of the intrinsic properties of the surface are expressible in terms of $\phi$, and for this reason $\phi$ is called the generating angle.

 The condition that there must be two distinct asymptotic lines at each point in $U$ implies that $\phi$ must satisfy the constraint 
\begin{equation}\label{angle-constraint}
0<\phi(u,v)<\pi.
\end{equation} 
Furthermore, it follows from Brioschi's formula \cite{gray} that $\phi$ must satisfy the sine-Gordon 
equation
\begin{equation}
\frac{\partial^2 \phi}{\partial u \partial v} =-K\sin(\phi), \label{Sine-Gordon}
\end{equation}
and thus solutions to the sine-Gordon satisfying (\ref{angle-constraint}) generate hyperbolic surfaces. In fact Hilbert proved that there is no smooth immersion of $\mathbb{H}^2$ by showing that there is no smooth solution to (\ref{Sine-Gordon}) that satisfies (\ref{angle-constraint}) in the entire $u\--v$ plane. The points on the surface where $\phi=n\pi$ are precisely the boundaries where the immersion fails to be smooth.

To determine the geodesics on $S$ we need to calculate the Christoffel symbols. They are intrinsic quantities determined by the metric and hence by $\phi$:
\begin{equation} \begin{array}{ll} \label{christoffel-symbols}
\Gamma_{11}^1=\cot(\phi)\frac{\partial \phi}{\partial u}, &  \Gamma_{11}^2=-\csc(\phi)\frac{\partial \phi}{\partial u}\\
\Gamma_{12}^1=0,  & \Gamma_{12}^2=0,\\  
\Gamma_{22}^1=-\csc(\phi)\frac{\partial \phi}{\partial v} &  \Gamma_{22}^2=\cot(\phi)\frac{\partial \phi}{\partial v},
\end{array}
\end{equation}
which gives us the following geodesic equations:
\begin{equation} \label{geodesic-equations}
\begin{array}{c}
\frac{d^2u}{dt^2}+\left(\frac{du}{dt}\right)^2\cot(\phi)\frac{\partial \phi}{\partial u}-\left(\frac{dv}{dt}\right)^2\csc(\phi)\frac{\partial \phi}{\partial v}=0,\\
\frac{d^2v}{dt^2}-\left(\frac{du}{dt}\right)^2\csc(\phi)\frac{\partial \phi}{\partial u}+\left(\frac{dv}{dt}\right)^2\cot(\phi) \frac{\partial \phi}{\partial v}=0.
\end{array}
\end{equation}

The extrinsic quantities of $S$ are also expressible in terms of $\phi$. The components of the second fundamental form $\mathbf{h}$ are given by
\begin{equation}
h_{11}=h_{22}=0 \text{ and } h_{12}=h_{21}=\pm \sin(\phi)
\end{equation}    
and the principal curvatures are
\begin{equation}
 k_1^2=\tan^2(\phi/2) \text{ and } k_2^2=\cot^2(\phi/2). \label{principal-curvatures}
\end{equation}
Consequently, since  $U\subset S$ is an isometric immersion of $D$ we have that equation (\ref{bending-energy-polar}) can be expressed entirely in terms of $\phi$:
\begin{equation}\label{bending-energy}
\mathcal{B}[\mathbf{x}]=\int_0^{2\pi}\int_0^R \sinh(r)(\tan^2(\phi/2)+\cot^2(\phi/2))\,dr\,d\Psi.
\end{equation}

We conclude this section by pointing out that a closely related parameterization of $S$, that is sometimes easier to make calculations with, is a parameterization by curves of vanishing curvature or a C-net of the second kind \cite{gray}. If $\mathbf{x}$ is a C-net, then a C-net of the second kind $\mathbf{y}$ is given by $\mathbf{y}(\eta,\xi)=\mathbf{x}\left(\frac{\eta+\xi}{2},\frac{\eta-\xi}{2}\right)$. The components of the metric $\mathbf{g}^{\prime}$ corresponding to $\mathbf{y}$ are
\begin{equation} \label{Tchebyshef-secondkind-metric}
g_{11}^{\prime}=\cos^2(\theta),\,g_{12}^{\prime}=g_{21}^{\prime}=0,\,g_{22}^{\prime}=\sin^2(\theta).
\end{equation}
Since $g_{12}^{\prime}=g_{21}^{\prime}=0$, the parameter curves of this parameterization are orthogonal and in fact they bisect the asymptotic curves. Therefore, $\theta=\phi/2$ and thus we can determine the generating angle by such a parameterization.

%%%%%%%%%%%%%%%%%%%%%%%%%%%%%%%%%%%%%%%%%%%%%%%%%%%%
%
%					Pseudosphere
%
%%%%%%%%%%%%%%%%%%%%%%%%%%%%%%%%%%%%%%%%%%%%%%%%%%
\section{Elastic Energy of the Pseudosphere} \label{sec:pseudosphere}

To illustrate the process of calculating the elastic energy of an isometric immersion of $D$ we consider a commonly known hyperbolic surface, the pseudosphere. The pseudosphere is given by the following C-net of the second kind \cite{gray}
\begin{equation}
\mathbf{y}(\eta,\xi)=\left(\frac{\cos(\xi)}{\cosh(\eta)},\frac{\sin(\xi)}{\cosh(\eta)},\eta-\tanh(\eta)\right).
\end{equation}
The mapping is singular on the curve $\eta=0$, which is where the surface fails to be a smooth immersion, and consequently we will only consider the ``upper half" $\eta>0$ of the pseudusphere.  Furthermore, since the sheets we are modeling have the topology of a plane and not a tubular surface we will think of the pseudosphere as a multi-sheeted surface. 

Directly calculating, the metric coefficients are given by the following simple expressions:
\begin{equation}
g_{11}^{\prime}=\tanh^2(\eta),\, g_{22}^{\prime}=\frac{1}{\cosh^2(\eta)}\text{ and } g_{12}^{\prime}=g_{21}^{\prime}=0.
\end{equation}
Consequently, by equation (\ref{Tchebyshef-secondkind-metric}) the generating angle has the form
\begin{equation}
\phi(\eta,\xi)=4\arctan(e^{-\eta}),
\end{equation}
and thus by (\ref{principal-curvatures}) the principal curvatures are
\begin{equation}
k_1^2=\tan^2(2\arctan(e^{-\eta}))=\frac{1}{\sinh^2(\eta)} \text{ and } k_2^2=\cot^2(2\arctan(e^{-\eta})=\sinh^2(\eta).
\end{equation}
Therefore, by equation (\ref{bending-energy}) the elastic energy of a disk $U$ centered at $(\eta_0,0)$ and lying in the pseudosphere is simply
\begin{equation}\label{ps-energy}
\mathcal{B}[\mathbf{y}]=\int_{0}^{2\pi}\int_0^{R}\sinh(r) \left(\sinh^2(\eta(r,\Psi))+\frac{1}{\sinh^2(\eta(r,\Psi))}\right)\,dr\,d\Psi.
\end{equation}  

Now, to determine how $\eta$ depends on $r$ and $\Psi$ and the shape of a disk $U$ we need to calculate the arclength of geodesics on the pseudosphere. The arclength of a curve $\xi(\eta)$ starting at $(\eta_0,0)$ and terminating at $(\eta_f,\xi_f)$ is given by the functional
\begin{equation} \label{arclength}
L[\xi(\eta)]=\int_{\eta_0}^{\eta_f}\sqrt{\tanh^2(\eta)+\sinh^2(\eta)\left(\frac{d\xi}{d\eta}\right)^2}\,d\eta
\end{equation}
and the geodesics that extremize this functional are given implicitly by
\begin{equation} \label{pseudo-geodesic-implicit}
\cosh^2(\eta)+(\xi+C)^2=D,
\end{equation}
where $C$ and $D$ are determined by the condition that the geodesic passes through $(\eta_0,0)$ and $(\eta_f,\xi_f)$.

Define the function $\mathcal{L}(\eta,\xi)$ to be arclength of the geodesic starting at $(\eta_0,0)$ and terminating at $(\eta,\xi)$. Since the geodesics fail to be a function of $\eta$ at the critical value $\left(\eta^{*}=\text{arccosh}(\sqrt{D}),\xi^{\star}=\sqrt{\cosh^2(\eta_0)-\cosh^2(\eta_f)}\right)$ the arclength of a geodesic will have to be computed over both branches of a square root. Assuming that $\eta_0<\eta$, if $\eta_0>\eta$ then $\mathcal{L}(\eta,\xi)$ can be computed by switching $\eta_0$ and $\eta$, we have by equation (\ref{arclength}) that
\begin{equation} \label{L(m,n)}
\mathcal{L}(\eta,\xi)=\left\{\begin{array}{ccc}
\int_{\eta_0}^{\eta}\sqrt{\frac{D\tanh^2(t)}{D-\cosh^2(t)}}\,dt &\text{ if } & 0<\xi^2< (\xi^*)^2\\
\int_{\eta_0}^{\eta^*}\left(\sqrt{\frac{D\tanh^2(t)}{D-\cosh^2(t)}}+\sqrt{\frac{D\tanh^2(t)}{D-\cosh^2(t)}}\right)\,dt &\text{ if } & \xi^2>(\xi^*)^2\\
\ln\left(\frac{\cosh(\eta)}{\cosh(\eta_0)}\right) & \text{ if } &  \xi=0 
 \end{array}\right..
\end{equation} 
By making the substitution $z=\text{sech}(t)$, and taking into account the case $\eta_0>\eta$, the integrals can be explicitly evaluated. The contour $\mathcal{L}(\eta,\xi)=R$ forms the boundary of the disk $U$, (see figure \ref{ps:pseudospherecut}).
\begin{figure}[htp] 
\begin{center}
\subfigure[]{
\includegraphics[width=.45\textwidth]{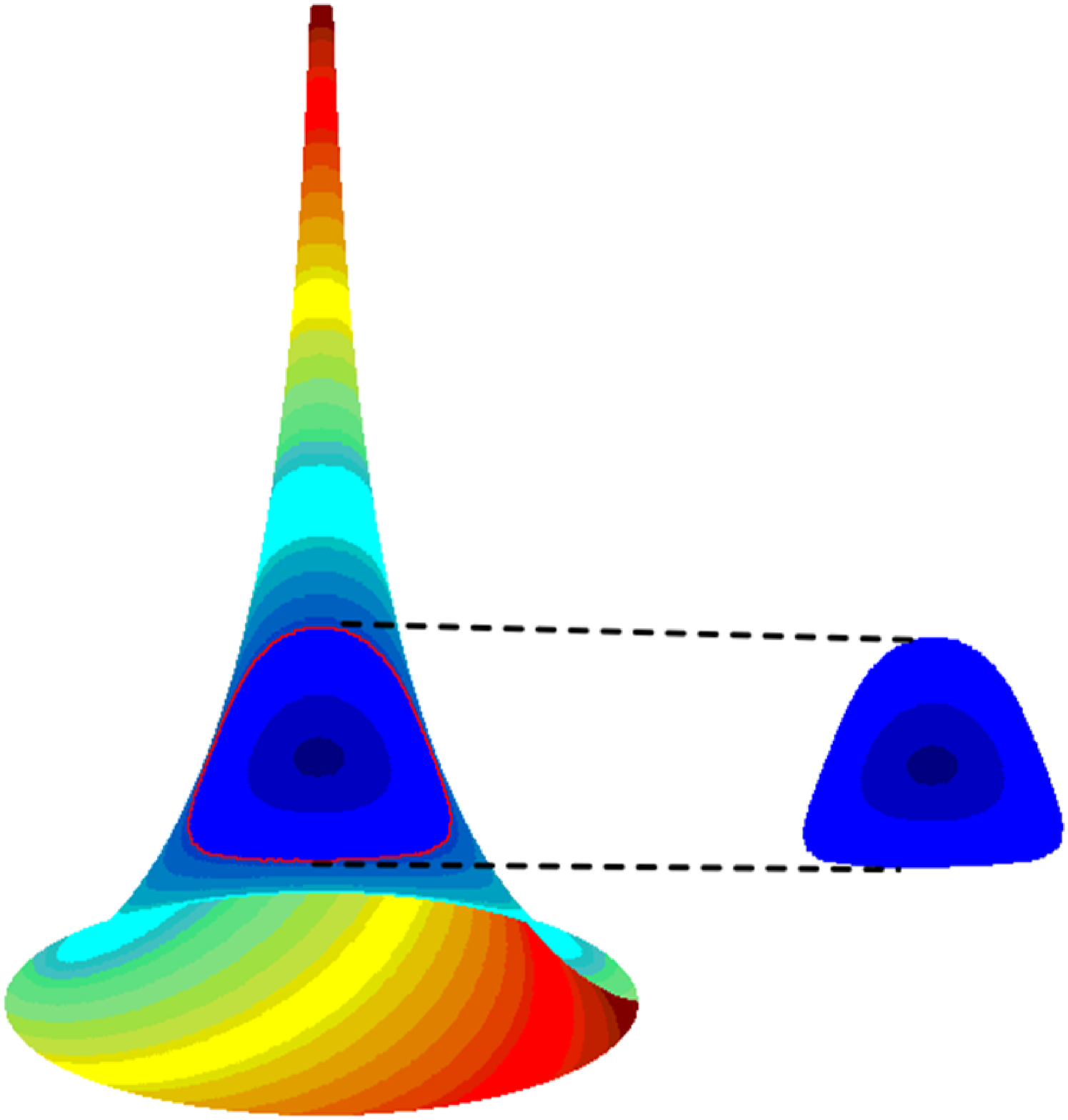}
\label{ps:pseudospherecut}
}
\subfigure[]{
\includegraphics[width=.45\textwidth]{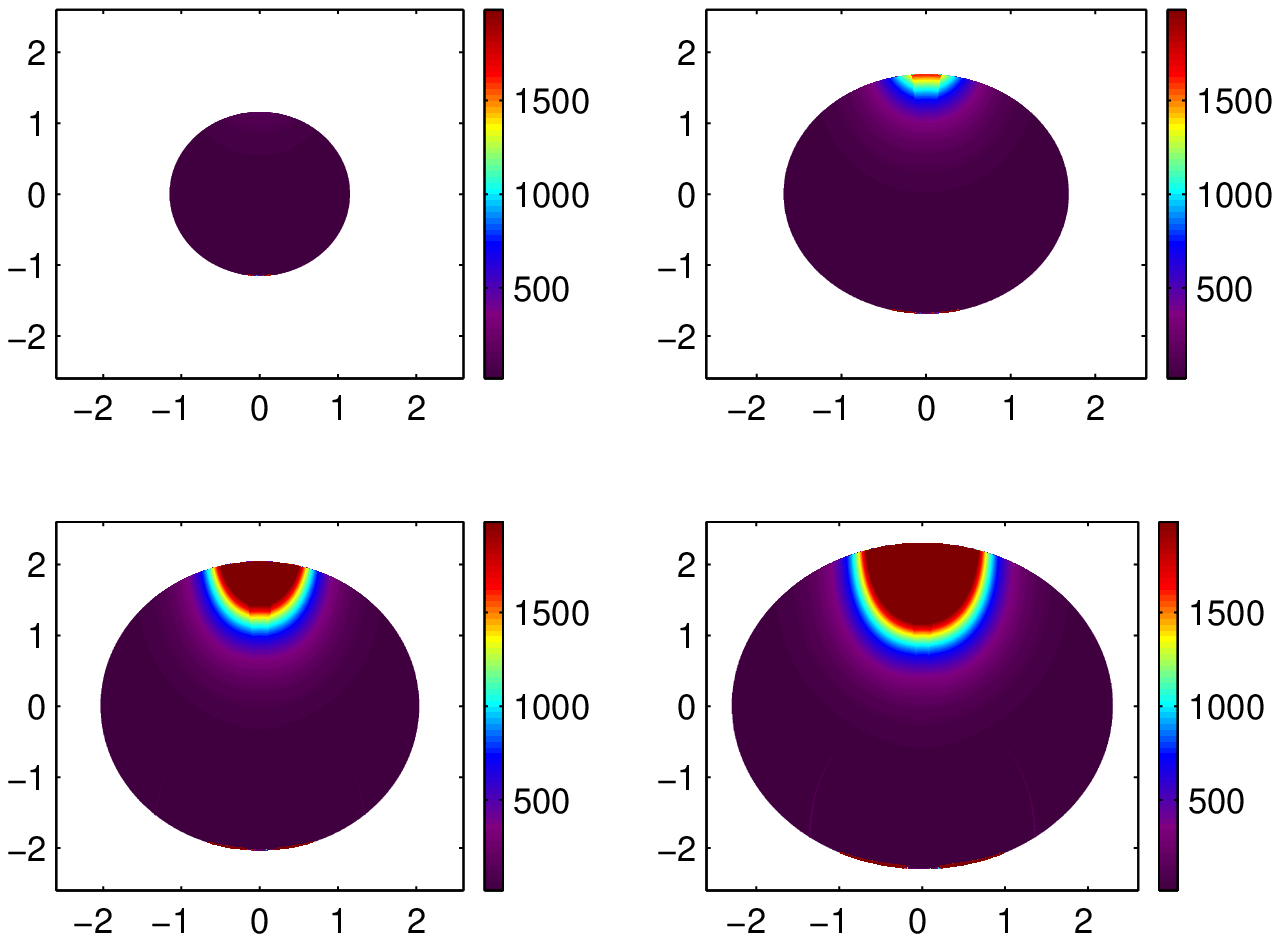}
\label{ps:psdiscs}
}

\subfigure[]{
\includegraphics[width=.8\textwidth]{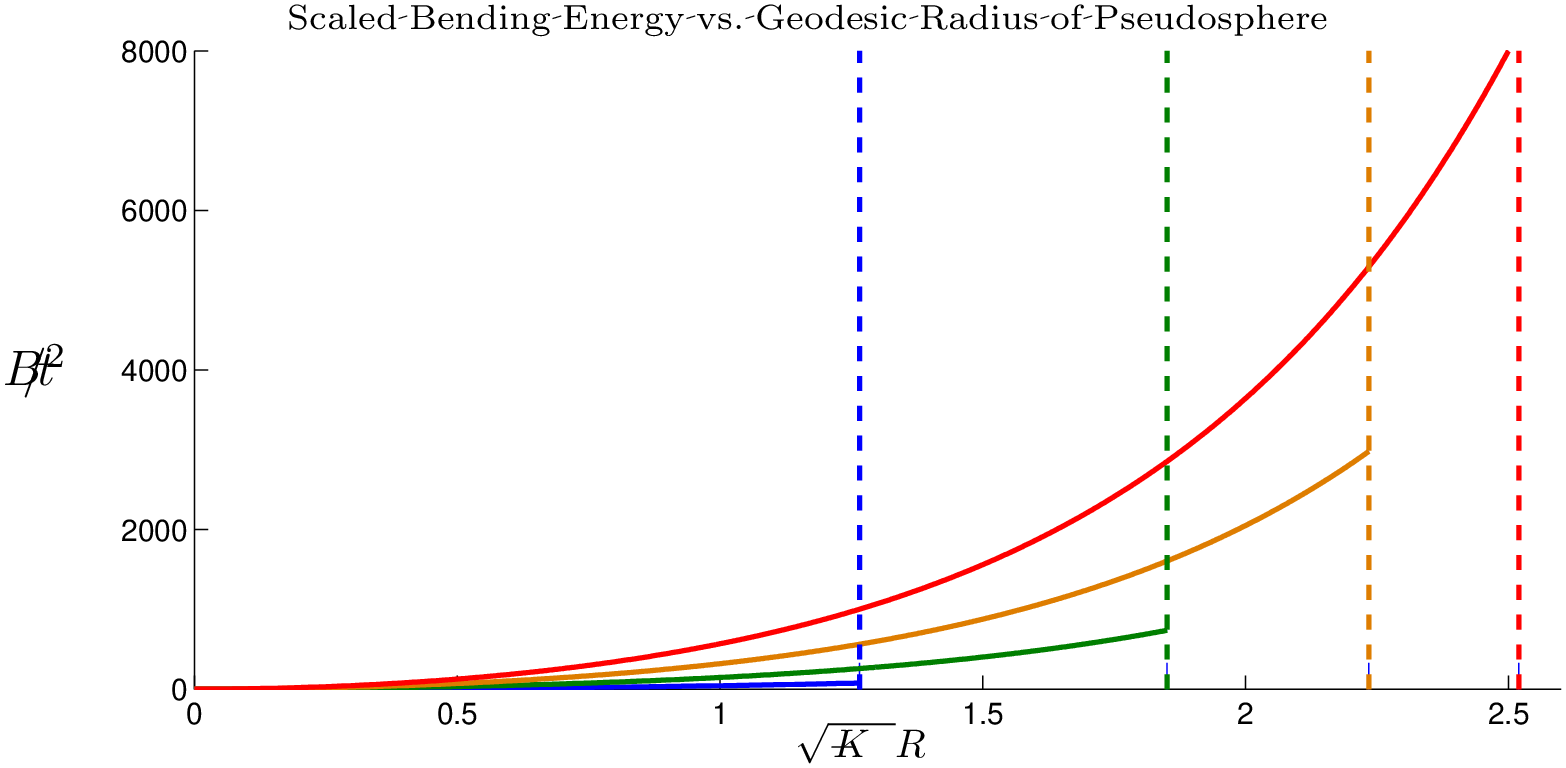}
\label{ps:energyplot}
}
\end{center}
\caption{\subref{ps:pseudospherecut} The pseudosphere colored by the geodesic distance from an arbitrary point $p$. Geodesic disks of constant negative curvature can be formed from the contour data and correspond to isometric immersions of geodesic disks in the hyperbolic plane. \subref{ps:psdiscs} A representation of the geodesic disks colored by $k_1^2+k_2^2$ in which the radius and polar angle correspond to the geodesic radius $r$ and polar angle $\Psi$. These disks are centered at $\eta_0=1.94,2.54, 2.92, 3.21$ respectively. These values where chosen for comparison with the periodic Amsler surfaces in figure \ref{amslerenergy}. \subref{ps:energyplot} Plot of the scaled bending energy of these disks as a function of the geodesic radius. We can see from figure \subref{ps:psdiscs} that for moderately sized disks the bending energy is large near the top of the disk leading to the very large energies. 
}
\label{psfigure}
\end{figure}

Furthermore, the geodesics on the pseudosphere itself are given by $
\mathbf{x}(\eta=\text{arccosh}\left(D-(\xi+C)^2\right),\xi)$.
The tangent vector $\frac{d\mathbf{x}(\eta(\xi),\xi)}{d\xi}$ evaluated at $\eta=0$ gives us the tangent vector of this curve which we can use to calculate the angle a geodesic makes with respect to the basis $\left(\frac{\partial \mathbf{x}(\eta,\xi)}{\partial \xi}, \frac{\partial \mathbf{x}(\eta,\xi)}{\partial \eta}\right)$ of the tangent plane at the center of the disk by taking inner products. By making the substitution $D=\cosh(\eta_0)^2+C^2$, this calculation gives us that the angle $\Psi$ satisfies  
\begin{equation}
\cos(\Psi)=\frac{\sqrt{2}\text{csch}(\eta_0)}{1+2C^2+\cos(2\xi_0)}.
\end{equation}
Consequently, by specifying a value for $\Psi$ we can calculate $C$ and therefore use equation (\ref{L(m,n)}) to determine $\eta$ for a particular value of $r$ and $\Psi$. This allows us to color the disk by bending energy density (see figure \ref{ps:psdiscs}) and numerically integrate (\ref{ps-energy}) over the disk to determine how the bending energy scales with $R$ for various values of $m_0$ (see figure \ref{ps:energyplot}).  

Now, a disk centered at $(\eta_0,0)$ cannot be made arbitrarily large since it will eventually meet the singular curve $\eta=0$. We can prove that the largest disk on the pseudosphere with this center has a radius of $R_{\eta_0}=\ln(\cosh(\eta_0))$. To prove this note that the curve $\xi=0$ is a geodesic, it trivially solves the Euler-Lagrange equation, and therefore
\begin{equation}
\mathcal{L}(0,0)=\int_0^{\eta_0}\tanh(u)\,d\eta=\ln(\cosh(\eta_0)).
\end{equation}
Now, let $\xi(\eta)$ be another geodesic that terminates at a point $(0,\eta^*)$ on the boundary. Then,
\begin{align*}
\mathcal{L}(0,\xi^*)&=\int_0^{\eta_0}\sqrt{\tanh^2(\eta)+\sinh^2(\eta)\left(\frac{d\xi}{d\eta}\right)^2}\,d\xi\\
&>\int_0^{\eta_0}\tanh(\eta)\,d\eta=\ln(\cosh(\eta_0)).
\end{align*}
Consequently, the contour $\mathcal{L}(\eta,\xi)=\ln(\cosh(\eta_0))$ meets the curve $\eta=0$ at only one point and thus $\ln(\cosh(\eta_0))$ is the radius of the largest possible disk centered at $(\eta_0,0)$. Therefore, we can see that for arbitrary $R$ we can find a value $\eta_0$ such that $D$ is isometric to $U$, and thus by varying the center of the disk we can create arbitrarily large stretching free configurations. This proves the following proposition
\begin{proposition}
Let $D$ be a disk of radius $R$ in the hyperbolic plane. There exists a smooth isometric immersion $\mathbf{x}:D\rightarrow U\subset \mathbb{R}^3$ such that $U$ is a subset of the pseudosphere.
\end{proposition}

%%%%%%%%%%%%%%%%%%%%%%%%%%%%%%%%%%%%%%%%%%%%%%%%%%%%%%%%
%
% Lower Bound for the Curvature of Smooth Isometric immersions of Hyperbolic Disks				
%
%%%%%%%%%%%%%%%%%%%%%%%%%%%%%%%%%%%%%%%%%%%%%%%%%%%%%%%%%%%%%%%%
\section{Lower Bounds for the Curvature of Smooth Isometric immersions of Hyperbolic disks}
\label{sec:minimax}

In this section we explore numerically how the principal curvatures of smooth isometric immersions scale with the size of the disk. In particular, we show that the maximum principal curvature on the disk grows exponentially in $\sqrt{-K_0}R$, where $K_0$ is the target Gaussian curvature.

Let $p$ denote the center of the disk, and let us choose the asymptotic coordinates $u$ and $v$ such that $p$ is the point $(0,0)$. If $|u_0|+|v_0| < R$, it follows easily that there exists $p'$ in $D$ with coordinates $(u_0,v_0)$ since $d(p,p') \leq |u_0|+|v_0| < R$. Let $q$, $r$, $s$ and $t$ denote the vertices of  a ``square" denoted $[qrst]$ in asymptotic coordinates given by the intersections of the asymptotic lines $u=-R/(2+\nu),v=-R/(2+\nu),u=R/(2+\nu),v=R/(2+\nu)$ for some given $\nu > 0$. 

We redefine the variables
$$
u' = \frac{1}{2}+\frac{(2+\eta)u}{2 R}, \quad  v' = \frac{1}{2}+\frac{(2+\eta)v}{2 R}, \quad  \lambda = \frac{4 |K_0|R^2}{(2+\eta)^2}, \quad \phi'(u',v') =\phi(u,v),
$$ and then drop the primes to transform equation (\ref{Sine-Gordon}) into 
\begin{equation}
\phi_{uv} = \lambda \sin(\phi) \quad \mbox{ on the domain } (0,1) \times (0,1)
\label{eq:sg}
\end{equation}
A smooth isometric immersion corresponds to a smooth function $\phi :(0,1)^2 \rightarrow (0,\pi)$ satisfying the Sine-Gordon equation (\ref{eq:sg}).

We want to consider the variational problem for minimizing the $L^p$ norms of the principal curvatures. The natural mode of converge is then $L^p$ convergence and this does not preserve the class of smooth functions. We thus enlarge the admissible set. Let $\mathcal{M}$ denote the set of all measurable functions on the unit square with values in $[0,\pi]$ let  $\mathcal{A} \subseteq \mathcal{M}$ be the set of measurable functions $\phi:[0,1]^2 \rightarrow [0,\pi]$ that are distributional solutions of the Sine-Gordon equation, {\em i.e} 
$$
\int_{[0,1]^2} \phi g_{uv} du dv = \lambda \int_{[0,1]^2} \sin(\phi) g du dv, \qquad \forall g \in C_c^\infty((0,1)^2).
$$
Clearly $\mathcal{A} \subseteq L^p([0,1]^2)$ for all $1 \leq p \leq \infty$. 
\begin{remark} If $\phi_n$ is a sequence in $\mathcal{A}$ and $\phi_n \rightarrow \phi$ in $L^p$, the mean value theorem implies that $\|\sin(\phi_n) - \sin(\phi)\|_p \leq \|\phi_n - \phi\|_p \rightarrow 0$, so it follows that $\phi \in \mathcal{A}$. Consequently, $\mathcal{A}$ is a closed subset of $L^p([0,1]^2)$ for all $1 \leq p \leq \infty$. \end{remark} 
\begin{remark} (Compactness) Since $\mathcal{M}$ consists of bounded measurable functions, for every sequence $\phi_n \in \mathcal{M}$, there exists a subsequence $\phi_{n_k}$ and a limit point $\phi^* \in \mathcal{M}$ such that $\phi_{n_k} \rightharpoonup \phi^*$ weakly in all $L^p([0,1])^2$ for $1 \leq p < \infty$ and $\phi_{n_k} \stackrel{*}{\rightharpoonup} \phi^*$ (converges weak-$*$) in $L^\infty([0,1]^2)$. Note that the limit $\phi^*$ is independent of $p$. In fact, on $\mathcal{M}$, the weak $L^p$ topologies for $1 \leq p < \infty$ and the weak-$*$ $L^\infty$ topoloy are all the same \cite{evans}.
\label{rem:cmpct}
\end{remark}
We are interested in a lower bound for the principal curvatures of a {\em Chebyshev patch}, {\em i.e.}, a portion of a surface of constant negative curvature covered by a single coordinate system given by a Chebyshev net. We thus define
$$
M(\lambda) = \inf_{\phi \in \mathcal{A}} \sup_{x \in [0,1] \times[0,1]} \max\left[\cot\left(\frac{\phi}{2}\right),\tan\left(\frac{\phi}{2}\right)\right].
$$
While we can estimate $M(\lambda)$ by direct numerical computation, we use a slightly different approach that yields some more information. Since $\cot(\phi/2)+\tan(\phi/2) = 2 \csc(\phi)$, we get $k_1^2+k_2^2 = 4 \csc^2(\phi) -2 = 4 \cot^2(\phi) + 2$. This motivates the definition
$$
I_p(\lambda) = \inf_{\phi \in \mathcal{A}} \|\cot^2(\phi)\|_p
$$
where $\|.\|_p$ denotes the $L^p$ norm.  This corresponds to the variational problem for  functionals $
F_p : \mathcal{M} \rightarrow [0,\infty]$ given by 
\[
F_p(\phi) = \begin{cases} \|\cot^2 \phi\|_p, & \phi \in \mathcal{A}. \\ + \infty, & \text{otherwise}. \end{cases}
\]
Note that $M(\lambda)$ and $I_p(\lambda)$ are for a unit square in asymptotic coordinates, and this is not the same thing as a geodesic disk, although the 
two can be related. 

For all $p < \infty$, the map $f \mapsto \|f\|_p$ is differentiable at $f \neq 0$. Also, for any function $f:[0,1]\times[0,1] \rightarrow \mathbb{R}$, the $p$-norm  $\|f\|_p$ is a nondecreasing function of $p$.  It is easy to see that, as a consequence,  $I_p$ is non-decreasing in $p$.

A natural question is whether there exists a minimizer for $F_p$. In this context we have
\begin{proposition} Let $\phi_n$ be a minimizing sequence for $F_p$, {\em i.e.}, $\phi_n \in \mathcal{A}$ and $F_p(\phi_n) \rightarrow I_p$. If $\phi_n$ (or a subsequence) converges pointwise to $\phi^*$, $\phi^*$ is a minimizer for $F_p$. 
\label{prop:existence}
\end{proposition} 
\begin{proof} 
If $\phi_n$ (or a further subsequence) converges pointwise, it follows from dominated convergence that $\|\phi_n - \phi^*\|_{p} \rightarrow 0$. Since $\mathcal{A}$ is closed in $L^p$, it follows that $\phi^* \in \mathcal{A}$. Also, by Fatou's lemma, we have 
$$
I_p \leq F_p(\phi^*) \leq \liminf F_p(\phi_n) = I_p
$$
showing that the minimum is attained. 
\end{proof}

\begin{remark}  The function $\theta \mapsto - \lambda \sin(\theta)$ is convex on $(0,\pi)$. Therefore, the map 
$$
 \mathcal{M} \ni \phi \mapsto - \lambda \int_{[0,1]^2} \sin(\phi) h du dv,\quad h \in C_c^\infty((0,1)^2), \, h \geq 0,
$$ 
is weakly lower semi-continuous. If $\phi_n \in \mathcal{A}$ is any sequence and $\phi_n \rightharpoonup \phi^*$ weakly in some $L^p, 1 \leq p < \infty$ or converges weak-$*$ in $L^\infty$, $\phi_n \stackrel{*}{\rightharpoonup} \phi^*$ it follows that
$$
\int_{[0,1]^2} \phi^* h_{uv} du dv - \lambda \int_{[0,1]^2} \sin(\phi^*) h du dv \leq 0, \qquad \forall h \in C_c^\infty((0,1)^2), \, h \geq 0.
$$
So, in general, $\phi^* \notin \mathcal{A}$, and thus, the functional $F_p$ are not  lower semi-continuous with respect to weak or weak-$*$ convergence. Thus, it is not enough to have weak convergence of the minimizing sequence for the existence of a minimizer. The pointwise convergence in the proposition is also necessary for using the  direct method to show the existence of a minimizer for $F_p$.
\end{remark}
The previous remark suggests considering the problem  of minimizing the relaxed energy $\tilde{F}_p$, namely the largest weakly lower semi-continuous function less than or equal to $F_p$ on $\mathcal{M}$ \cite{pedregal:yngmsr,muller}.  Let $\mathcal{B}$ denote the weak closure of $\mathcal{A}$ in some $L^{p'}$ (the precise value of $p'$ is irrelevant as indicated in Remark~\ref{rem:cmpct}). The relaxation of $F_p$ is defined by
\[
\tilde{F}_p(\phi) = \begin{cases} \inf\{ \liminf_n F_p(\phi_n) : \phi_n \rightharpoonup \phi\}  & \phi \in \mathcal{B}. \\ + \infty, & \mbox{otherwise} \end{cases}
\]
where the infimum is over all sequences which converge weakly to $\phi$. Clearly, $\mathcal{A} \subseteq \mathcal{B} \subseteq \mathcal{M}$ and $\tilde{F}_p$ is weakly lower semi-continuous by construction. It thus follows from Remark~\ref{rem:cmpct} that we can prove the existence of a minimizer for $\tilde{F}_p$ by the direct method in the calculus of variations.

The question of the existence of a minimizer for $F_p$ can thus be posed in terms of the nature of the minimizer of $\tilde{F}_p$ which always exists. If there exists $\phi^* \in \mathcal{A}$ which is a minimizer for $\tilde{F}_P$, then $\phi^*$ is also a minimizer for $F_p$. Conversely, if all the minimizers of $\tilde{F}_p$ are in $\mathcal{B} \setminus \mathcal{A}$, it follows that $F_p$ does not have a minimizer. Studying this question involves computing the relaxation $\tilde{F}_p$ which in turn needs a careful analysis of the potential oscillations in  weakly convergent sequences of solutions of the Sine-Gordon equation. We are investigating these issues in ongoing work.

 \subsection{A numerical investigation of lower bounds for the curvature}
 
 We can discretize the Sine-Gordon equation on a $N \times N$ grid covering the unit square, and minimize the $l^p$ norms of $\cot^2 \phi$ over all solutions of this discrete Sine-Gordon equation. This is a finite dimensional problem with a coercive energy, so there always exists a minimizer $\phi^N$ on the $N \times N$ grid. Fig.~\ref{numerics:lpnorm} displays the numerically obtained results for this minimization. Observe the lack of high frequency oscillations in the numerically obtained minimizers. In fact, increasing the mesh size suggests that as $N \rightarrow \infty$, the numerically obtained minimizers converge pointwise on $[0,1]^2$. If this were indeed the case, Proposition~\ref{prop:existence} implies the existence of minimizers for the functionals $F_p$ on the admissible set $\mathcal{A}$.
 
We will henceforth assume that the infimum is attained, {\em i.e.},  for any given  $p$ and $\lambda$, there is an admissible function $\phi_p(\lambda) \equiv \phi_p \in \mathcal{A}$ such that $I_p = \|\cot^2(\phi_p)\|_p$, it follows that 
$$
I_p \leq I_\infty \leq \|\cot^2(\phi_p)\|_{\infty}.
$$
Thus, numerically determining the function $\phi_p$ which minimizes the $p$-norm of $\cot^2(\phi)$ in the admissible set, gives both lower and upper bounds for $I_\infty$, and the difference between these bounds gives an estimate for the error in a numerical determination of $I_\infty$.

Once we determine $I_\infty(\lambda)$, we can compute $M(\lambda)$ as the following argument shows:
\begin{align*}
\max(|k_1|,|k_2|) & = \frac{|k_1|+|k_2|}{2} +   \frac{||k_1|- |k_2||}{2} \\
& =  \sqrt{\frac{k_1^2+2 |k_1 k_2| + k_2^2}{4}} +  \sqrt{\frac{k_1^2 - 2| k_1 k_2| + k_2^2}{4}} \\
& = \sqrt{\frac{k_1^2+2 + k_2^2}{4}} +  \sqrt{\frac{k_1^2 - 2 + k_2^2}{4}} \\
& = \sqrt{\cot^2(\phi)+1}+ \sqrt{\cot^2(\phi)} \equiv g(\cot^2(\phi))
\end{align*}
where $g(x) = \sqrt{x} + \sqrt{1+x}$ is a monotone function with a monotone inverse. Consequently, 
$$
M(\lambda) = g(I_\infty(\lambda)) = \sqrt{I_\infty(\lambda)} + \sqrt{I_\infty(\lambda)+1},
$$
and the value of $M(\lambda)$ is attained on $\phi_\infty$, the extremizer for $I_\infty$. 

The same argument can be applied to any function $\zeta(\phi)$ with the property that there is a continuous monotone function $h$ with a continuous monotone inverse such that $h(\zeta(\phi)) = \cot^2(\phi)$. In particular, if $\zeta(\phi) = (\phi - \pi/2)^2$, there is such a function $h$, and it follows that 
\begin{align}
\phi_\infty & \equiv \mbox{arg}\min_{\phi \in \mathcal{A}} \max_{x \in D} \cot^2(\phi) \nonumber \\
& = \mbox{arg}\min_{\phi \in \mathcal{A}} \max_{x \in D} \max[\tan(\phi/2),\cot(\phi/2)] \nonumber \\
& = \mbox{arg}\min_{\phi \in \mathcal{A}} \max_{x \in D} \left[\phi - \frac{\pi}{2}\right]^2.
\label{eq:equivalence}
\end{align}
We can solve any of these minimax problems to determine the extremizing function $\phi_\infty$.  For our numerics, we use the third formulation of the problem, {\em i.e.}, finding the minimax for the square of the deviation of $\phi$ from $\pi/2$. Also, we note that, $I_\infty$ directly estimates the maximum bending energy density $k_1^2+k_2^2$ and $I_1$ estimates the total bending energy. 

Figure \ref{numerics:lpnorm} shows the numerically computed values of $I_p$ and $\|\cot^2(\phi_p)\|_\infty$ as a function of $p=2^n$, $n\in \mathbb{N}$, for $\lambda = 4$. Figure \ref{numerics:phi} and figure \ref{numerics:cot} shows the numerically determined extremizers $\phi_p$ as well as the quantity $\cot^2(\phi_p)$ for $\lambda = 4$ and $p = 16$. The figures for other values of $\lambda$ are similar in character. The key observation from these figures is that as $p$ gets large $\phi_p(u,v)$ numerically approaches a limit that is purely a function of the combination $u+v$ (see figure \ref{numerics:collapse}). This motivates the ansatz $\phi_\infty(u,v) = \psi(u+v)$. 

\begin{figure}[htp] \label{fig:numerics}
\begin{center}
\subfigure[]{
\includegraphics[width=.7\textwidth]{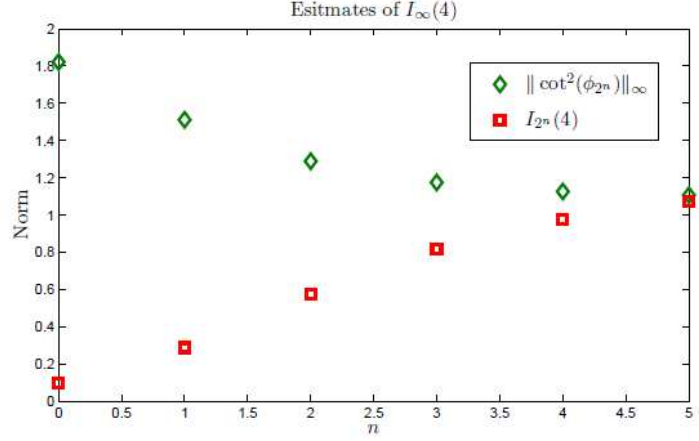}
\label{numerics:lpnorm}
}

\subfigure[]{
\includegraphics[width=.35\textwidth]{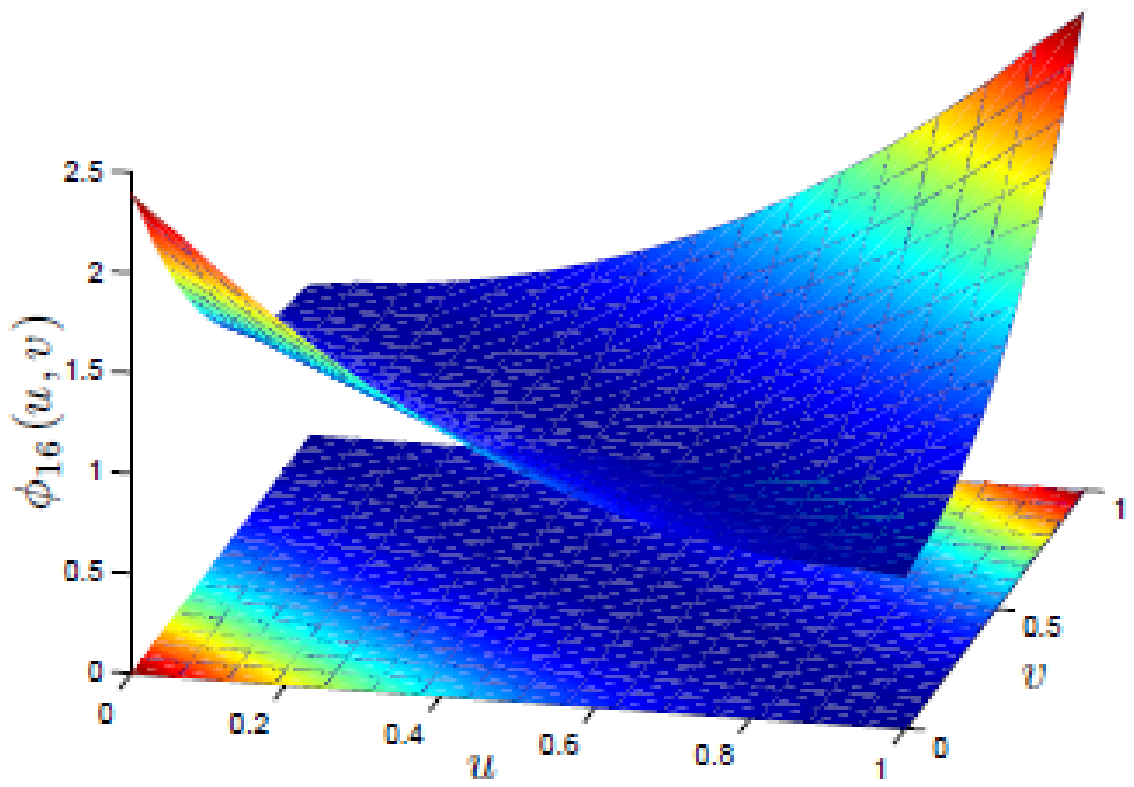}
\label{numerics:phi}
}
\subfigure[]{
\includegraphics[width=.35\textwidth]{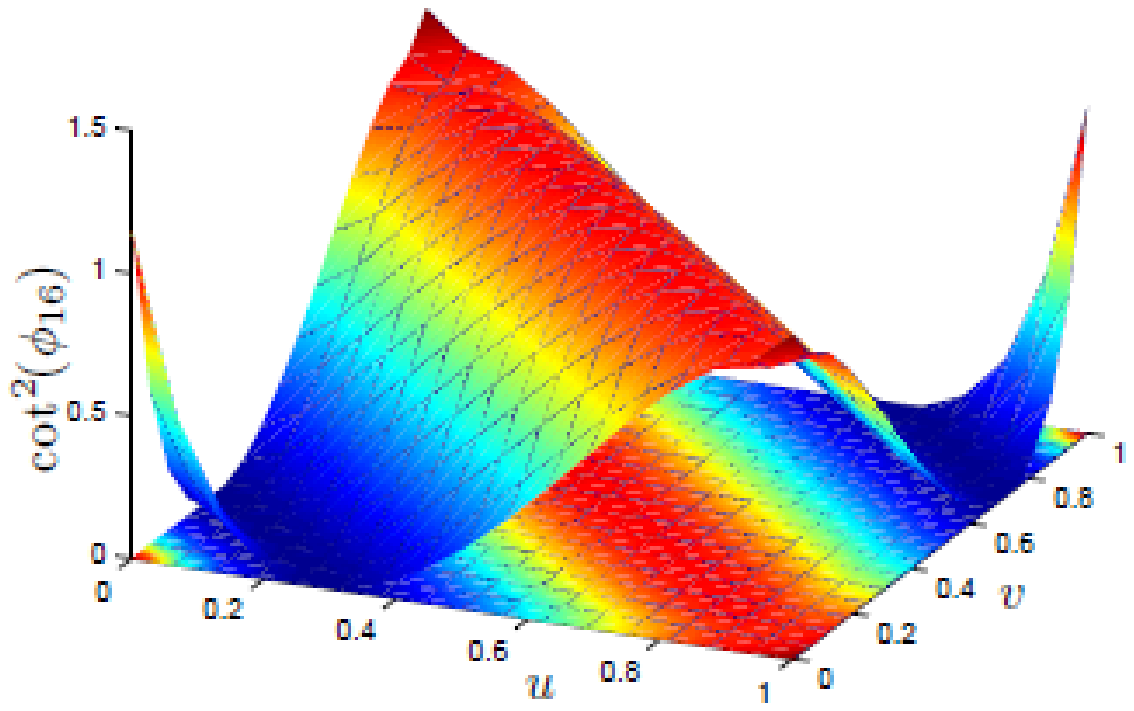}
\label{numerics:cot}
}

\subfigure[]{
\includegraphics[width=.7\textwidth]{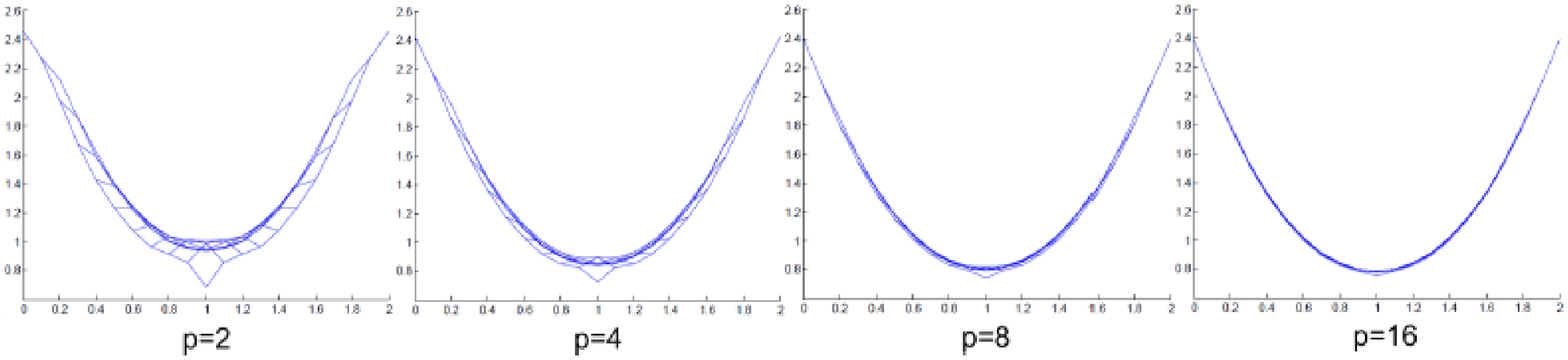}
\label{numerics:collapse}
}
\end{center}
\caption{\subref{numerics:lpnorm} Plot of $I_{2^n}$ and $\|\cot^2(\phi_{2^n})\|$ as a function of $n$ with $\lambda=4$. The collapse of these data points onto a single line confirms that these two quantities are valid approximations to $I_{\infty}$ for large enough $n$. \subref{numerics:phi} Plot of $\phi_{16}$ as a function of $u$ and $v$ with the contour plot of $\phi_{16}$ underneath the surface. We can see from the contour plot that it appears that lines $u+v=\text{constant}$ form the contours suggesting that $\phi_{\infty}$ is a function of $u+v$. \subref{numerics:cot} $\cot^2(\phi_{16})$ plotted as a function of $u$ and $v$ again with the contour plot $\cot^2(\phi_{16})$ plotted underneath the surface.  \subref{numerics:collapse} Plot of $\phi_p(u,v)$ as a function of $u+v$ for $p=2,4,8,16$ the collapse into a single curve again indicates that $\phi_{\infty}$ is a function of $u+v$ only. 
}
\end{figure}

 Substituting this ansatz in equation~(\ref{eq:sg}) yields $\psi'' = \lambda \sin(\psi)$, which is the Hamiltonian motion of a unit mass in a potential $V(\psi) = \lambda \cos(\psi)$. The trajectories in phase space $(\psi,\psi')$ are given by  conservation of the energy, $\psi'^2/2 + \lambda \cos(\psi) = E$ (see figure \ref{fig:phase_portrait}). Note that there are three distinct types of surfaces of revolution of constant negative curvature \cite{Rozendorn-1989} and for each of these surfaces $\phi$ is a function of $u+v$ only. Therefore, the three distinct trajectories in the phase plane correspond to the three types of surfaces of revolution. The pseudosphere corresponds to the separatrix, hyperboloid surfaces correspond to closed orbits, and conical surfaces correspond to unbounded orbits.
\begin{figure}[htp] \label{fig:phase_portrait}
\begin{center}
\includegraphics[width=.7\textwidth]{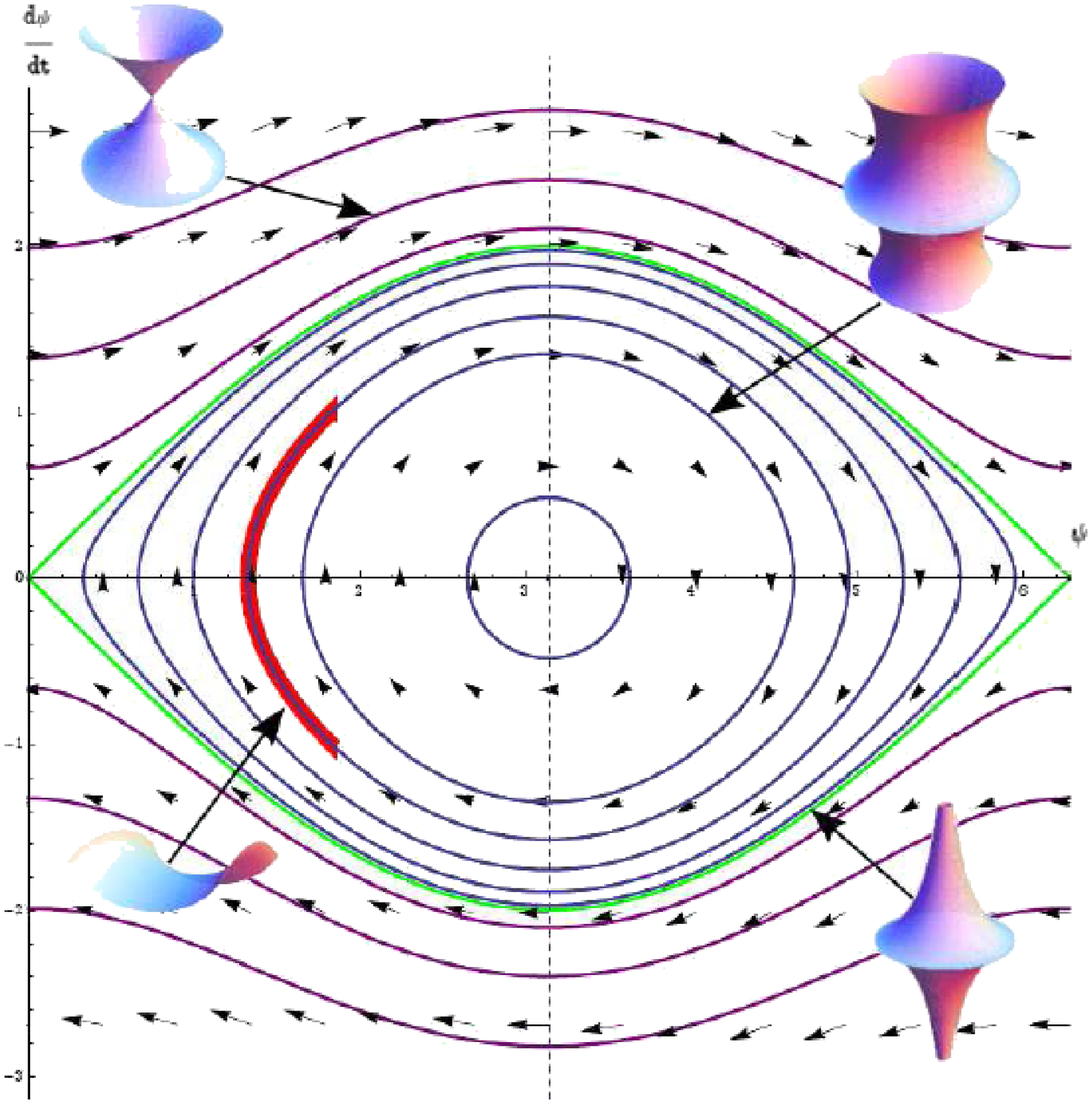}
\caption{Phase portrait of the differential equation $\psi^{\prime \prime}=\lambda \sin(\psi)$. The three different types of orbits correspond to the three different types of surfaces of revolution of constant negative curvature. The piece of the orbit outlined in bold is the trajectory $\psi:[0,2]\rightarrow (0,\pi)$ that minimizes the value of $\cot^2(\psi)$ over all trajectories in the phase plane defined over the same time domain.}
\end{center}
\end{figure}

We seek  solutions on the unit square, $0 \leq t \equiv u+v \leq 2$, that are bounded away from $\psi = \pi$ on the interval $[0,2]$. If we want trajectories that minimize the maximum value of $\cot^2(\psi)$, it is clear from the symmetry of $\cot(\psi)$ we are interested in the closed orbits. Furthermore, we seek a piece of these trajectories that satisfies the further symmetry requirement $\psi(0) = \psi(2) = \pi - \epsilon, \psi(1) = \epsilon, \text{ and } \psi^{\prime}(1)=0$ since if these conditions are not met we could select a different piece of the trajectory that could further minimize $\cot^2(\phi)$. This trajectory is depicted in bold in figure \ref{fig:phase_portrait}. 

The value $E$ of a minimizing trajectory can be computed as $E= [\psi'(1)]^2/2 + \lambda \cos(\psi(1)) = \lambda \cos(\epsilon)$, and the relation between $\epsilon$ and $\lambda$ is determined by requiring that the time it takes to get from $\psi = \pi - \epsilon $ to $\psi = \epsilon$ is 1, {\em i.e.},
\begin{equation}
\int_\epsilon^{\pi - \epsilon} \frac{d\theta}{\sqrt{2 \lambda(\cos(\epsilon) -\cos(\theta))}} = 1
\label{eq:tof}
\end{equation}
We can solve this equation numerically, and thereby determine $\phi_{\infty}(u,v) = \psi(u+v)$, provided that our ansatz is valid. 

Generating an asymptotic series for (\ref{eq:tof}), we have for small $\epsilon$ that
\begin{equation*}
\frac{-\log(\epsilon) + \log(8)+O(\epsilon)}{\sqrt{\lambda}} = 1.
\end{equation*}
Rearranging, we obtain
\begin{equation}
\cot(\epsilon) \sim \epsilon^{-1} + O(\epsilon) = \frac{1}{8} \exp[\sqrt{\lambda}] + O[e^{-\sqrt{\lambda}}] 
\end{equation}
This implies $I_{\infty}(\lambda) \sim \frac{1}{64} \exp[2 \sqrt{\lambda}] + O(1)$. Finally, if we set $\nu=0$ we have in terms of $\sqrt{-K_0}R$ that
\begin{equation} \label{eq:sclng}
I_{\infty}(\sqrt{-K}R)\sim \frac{1}{64}e^{2\sqrt{-K_0}R}+O(1).
\end{equation}  

\subsection{Elastic energy of hyperboloids of revolution}
We conclude this section by comparing the bending energy of disks cut from hyperboloids of revolution with the pseudosphere. The family of hyperboloids are parameterized by the following family of C-nets of the second kind
\begin{equation} \label{hyperboloid}
\mathbf{y}(\eta,\xi)=\frac{1}{b}\left(\text{dn}(\eta,b^2)\cos(bv),\text{dn}(\eta,b^2)\sin(b\xi),u-E(\text{am}(\eta,b^2)|b^2)\right),
\end{equation}
where $\text{dn}$, $\text{am}$, $\text{sn}$ denote the usual Jacobi elliptic functions \cite{AS}, $E$ is the elliptic function of the second kind \cite{AS}, and $0<b<1$ \cite{gray}. The metric and principal curvatures are given by
\begin{equation}\label{hyperboloid-metric}
g_{11}^{\prime}=b^4\text{sn}^2(\eta,b^2),\, g_{12}^{\prime}=g_{21}^{\prime}=0,\, g_{22}^{\prime}=1-b^4\text{sn}^2(\eta,b^2)
\end{equation}
\begin{equation}\label{hyperbolid-principal}
k_1^2=\frac{g_{22}^{\prime}}{g_{11}^{\prime}}=\frac{1-b^4\text{sn}^2(\eta,b^2)}{b^4\text{sn}^2(\eta,b^2)},\, k_2^2=\frac{b^4\text{sn}^2(\eta,b^2)}{1-b^4\text{sn}^2(\eta,b^2)}.
\end{equation} 
The Christoffel symbols for this parameterization are
\begin{equation} \begin{array}{ll} \label{christoffel-symbols-hyperboloid}
\Gamma_{11}^1=\frac{\text{cn}(\eta,b^2)\text{dn}(\eta,b^2)}{\text{sn}(\eta,b^2)}, &  \Gamma_{11}^2=0\\
\Gamma_{12}^1=0,  & \Gamma_{12}^2=\frac{-b^4\text{sn}(\eta,b^2)\text{cn}(\eta,b^2)\text{dn}(\eta,b^2)}{1-b^4\text{sn}(\eta,b^2)},\\  
\Gamma_{22}^1=-\frac{\text{cn}(\eta,b^2)\text{dn}(\eta,b^2)}{\text{sn}(\eta,b^2)} &  \Gamma_{22}^2=0,
\end{array}
\end{equation}
Therefore, the geodesic equations are
\begin{align}\label{hyperboloid-geodesics}
0&=\frac{d^2\eta}{dt^2}+\frac{\text{cn}(\eta,b^2)\text{dn}(\eta,b^2)}{\text{sn}(\eta,b^2)}\left[\left(\frac{d\eta}{dt}\right)^2+\left(\frac{d\xi}{dt}\right)^2\right]\\
0&=\frac{d^2\xi}{dt^2}+\frac{2b^4\text{sn}(\eta,b^2)\text{cn}(\eta,b^2)\text{dn}(\eta,b^2)}{1-b^4\text{sn}^2(\eta,b^2)}.
\end{align}

It is easy to see that the principal curvatures diverge when $\pi=\text{am}(\eta,b^2)$. Consequently, the center of the disk is located at the coordinate $\eta_0=F(\pi/2,b^2)$, $F$ is the elliptic function of the first kind, and thus the relationship between $\epsilon$ and $b$ is given by
\begin{equation}
\frac{\epsilon}{2}=\arctan\left(\sqrt{\frac{1-b^4\sin^2(\text{am}(F(\pi/2,b^2))}{b^4\sin^2(\text{am}(F(\pi/2,b^2))}}\right)=\arctan\left(\sqrt{\frac{1-b^4}{b^4}}\right).
\end{equation}
Solving this equation yields
\begin{equation}
b=\sqrt{\cos(\epsilon/2)}.
\end{equation}
Therefore, for a particular radius $R$ we can numerically solve equation (\ref{eq:tof}) to obtain $b$. Once, we determine $b$ we can then numerically integrate the bending energy
\begin{equation} \label{hyperboloid-energy}
\mathbf{\mathcal{B}[y]}=\int_{0}^{2\pi}\int_0^R\left(\frac{1-b^4\text{sn}^2(\eta,b^2)}{b^4\text{sn}^2(\eta,b^2)}+\frac{b^4\text{sn}^2(\eta,b^2)}{1-b^4\text{sn}^2(\eta,b^2)}\right)\sinh(r)\,drd\Psi,
\end{equation}
following a similar procedure as we did with the pseudosphere. These results are summarized in figure \ref{fig:hyperboloid}.
\begin{figure}[htp]\label{fig:hyperboloid}
\begin{center}
\subfigure[]{
\includegraphics[width=.8\textwidth]{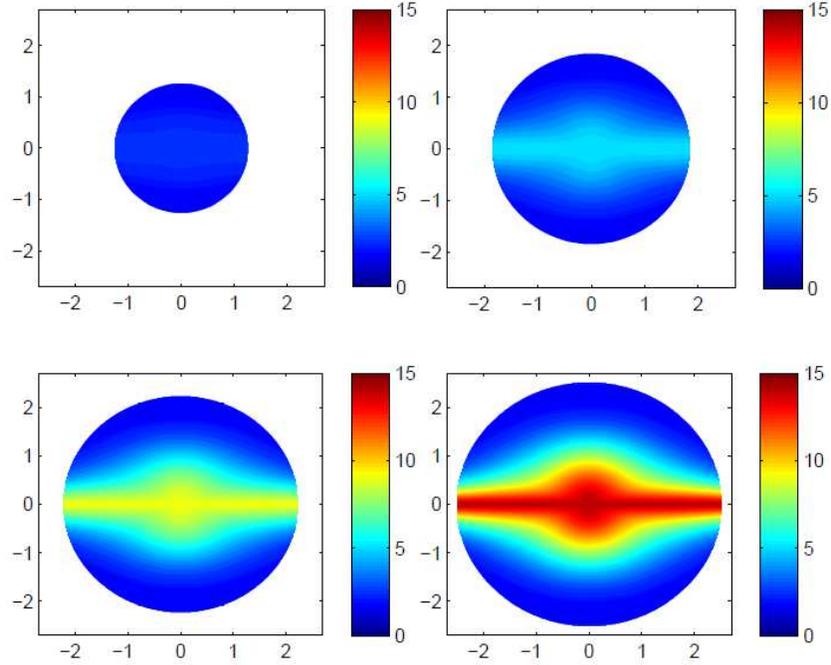}
\label{fig:hyperboloid:contour}
}
\subfigure[]{
\includegraphics[width=.8\textwidth]{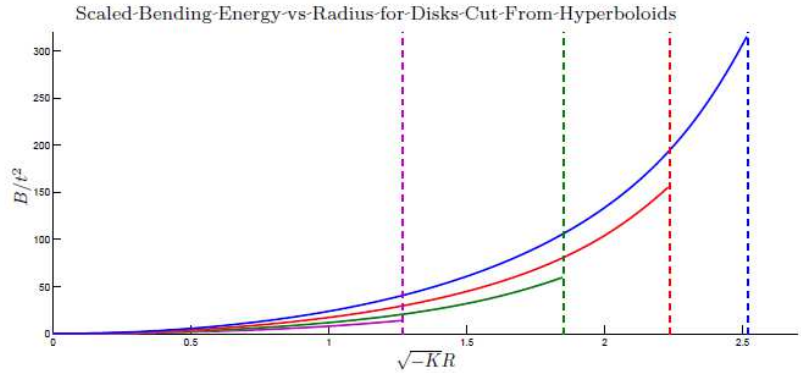}
\label{fig:hyperboloid:energy}
}
\caption{\subref{fig:hyperboloid:contour} A representation of the geodesic disks cut out of hyperboloid surfaces colored by $k_1^2+k_2^2$. The radius and polar angle of these disks correspond to the geodesic radius $r$ and polar angle $\Psi$. These disks each have a maximum radii of $R=1.2654,\, 1.8505,\, 2.2342,\, 2.5199$. These values were chosen for comparison with figures \ref{psfigure} and \ref{amslerenergy}.   \subref{fig:hyperboloid:energy} The scaled bending energy of disks cut from hyperboloids plotted versus geodesic radius.}
\end{center}
\end{figure}

%%%%%%%%%%%%%%%%%%%%%%%%%%%%%%%%%%%%%%%%%%%%%%%%%%%%%%%
%
%					 Small Slopes Approximation
%
%%%%%%%%%%%%%%%%%%%%%%%%%%%%%%%%%%%%%%%%%%%%%%%%%%%%%%%%
\section{Small Slopes Approximation}
\label{sec:small_slopes}

The configurations formed from the pseudosphere and hyperboloids can have very large bending energy concentrated at the center and top edge of the disk and these configurations do not have periodic profiles.  In order to understand the experimental observations for negatively curved disks \cite{an-experimental-study-of-shape-transitions-and-energy-scaling-in-thin-non-euclidean-plates},
we thus need to identify periodic configurations with a lower elastic energy.

 We begin by considering the small slopes approximation.  In this approximation, valid for $\epsilon=\sqrt{-K_0}\ll 1$, the configuration is assumed to have a parameterization $\mathbf{x}:D\rightarrow 
 \mathbb{R}^3$ of the form:
\begin{equation} \label{LinearGeometry}
\mathbf{x}(u,v)=(u+\epsilon^2 f(u,v),v+\epsilon^2g(u,v),\epsilon \omega(u,v)),
\end{equation}
where $f,g\in W^{1,2}(D)$, $\omega\in W^{2,2}(D)$  and $u,v$ are Cartesian coordinates on $D$ \cite{Landau, gamma_convergence_smallslopes, HaiyiLiang12292009}. That is, as a mapping,  $\mathbf{x} = i +  \epsilon \omega + \epsilon^2 \chi$, where $i$ is the standard immersion $\mathbb{R}^2 \rightarrow \mathbb{R}^3$, $\omega$ is the ``out of plane" deformation, {\em i.e.}, $\omega$ maps into the orthogonal complement of $i(\mathbb{R}^2)$ and $\chi = i\circ(f,g)$ is the ``in-plane" displacement. The target metric $\mathbf{g}$ takes the form 
\begin{eqnarray}
ds^2=\mathbf{g} &=& dr^2+\left(r^2+\frac{\epsilon^2 r^4}{3}\right)\,d\Psi^2 = (dr^2 + r^2 d\Psi^2) + \epsilon^2 \frac{r^4 d\Psi^2}{3}\\
&=& du^2+dv^2+\epsilon^2\left(-\frac{v^2}{3}du^2+2\frac{uv}{3}dudv-\frac{u^2}{3}dv^2\right).
\end{eqnarray}
 That is, $\mathbf{g} = \mathbf{g}_0 + \epsilon^2 \mathbf{g}_1$ is an asymptotic expansion for the target metric where $\mathbf{g}_0$ is the Euclidean metric on $\mathbb{R}^2$. 
 
 If we define the in-plane strain terms by
\begin{equation}
\begin{cases}
\gamma_{11}=2\frac{\partial f}{\partial u}+\left(\frac{\partial \omega}{\partial u}\right)^2 -\frac{y^2}{3}\\
\gamma_{12}=\frac{\partial f}{\partial v}+\frac{\partial g}{\partial u}+\frac{\partial \omega}{\partial u}\frac{\partial \omega}{\partial v}+\frac{uv}{3}\\
\gamma_{22}=2\frac{\partial g}{\partial v}+\left(\frac{\partial \omega}{\partial v}\right)^2-\frac{u^2}{3}
\end{cases}
\end{equation}
then using the above asymptotic expansions in the elastic energy (\ref{energy}) and keeping the lowest order terms, we get
\begin{align}
\bar{\mathcal{E}}[\mathbf{x}] & = \epsilon^{-2} \mathcal{E}[\mathbf{x}] \nonumber \\
& = \int_{B(R)}\left(\|D\chi^T + D \chi + D \omega^T \cdot D\omega-\mathbf{g}_1\|^2+t^2\|D^2\omega\|^2\right)\,dA \,+\mathcal{O}(\epsilon^2,t^0)\\
&= \int_{B(R)}\left(\gamma_{11}^2+2\gamma_{12}^2+\gamma_{22}^2\right)\,dudv+t^2\int_{B(R)}\|D^2\omega\|^2\,dudv \,+\mathcal{O}(\epsilon^2,t^0),
\label{reduced-energy}
\end{align}
where $B(R)$ is a (2 dimensional Euclidean) disk of radius $R$ and $dA$ is the Euclidean area element. The stretching and normalized bending energies under this approximation take the forms
\begin{equation}
S[\mathbf{x}]=\int_{B(R)}\left(\gamma_{11}^2+2\gamma_{12}^2+\gamma_{22}^2\right)\,dudv
\end{equation}
\begin{equation}\label{smallslopes-energy}
\mathfrak{B}[\mathbf{x}] = \int_{B(R)}\|D^2 \omega\|^2\,du\,dv,
\end{equation}
with the appropriate elastic energy given by
\begin{equation}
E_t[\mathbf{x}]=S[\mathbf{x}]+t^2\mathfrak{B}[\mathbf{x}].
\end{equation}
We are using the symbols $S$ and $\mathfrak{B}$ to indicate that these are the lowest order contributions to the stretching and bending energies. 

In this approximation, the Mean and Gaussian curvatures to the lowest orders in $\epsilon$ are given by
$$
H = \epsilon \, \, \mathrm{Tr}[D^2 \omega], \quad \quad K = \epsilon^2  \,\,\mathrm{Det}[D^2 \omega]
$$
so that a necessary and sufficient condition for the surface to have constant negative curvature $K = -\epsilon^2$ is that the out of plane displacement $\omega$ should satisfy the following hyperbolic Monge-Ampere equation
\begin{equation} \label{Monge-Ampere}
\frac{\partial^2 \omega}{\partial u^2}\frac{\partial^2 \omega}{\partial v^2}-\left(\frac{\partial^2 \omega}{\partial u \partial v}\right)^2=-1.
\end{equation}
Once a solution to (\ref{Monge-Ampere}) is chosen we can match the metric $\mathbf{g}^{\prime}$ to the target metric $\mathbf{g}$ by solving for the in plane stretching terms $f$ and $g$.

A class of solutions to (\ref{Monge-Ampere}) is the following one parameter family of quadratic surfaces
\begin{equation}
\omega_a(u,v)=\frac{1}{2}\left(au^2-\frac{1}{a}v^2\right),
\label{quad-surface}
\end{equation}
parametrized by $a\in \mathbb{R}^+$ \cite{Handbook-of-partial-differential-equations}. These solutions are saddle shaped such that $\omega$ vanishes on a pair of lines passing through the origin of the $u\--v$ plane intersecting at an angle $\theta=\pm \arctan(a)$. 

To create periodic structures we first consider the function $\omega_a^*(\theta,r)=\omega_a(r,(\theta-\arctan(a))$ which is simply a rotation of $\omega_a$ that aligns the asymptotic lines with the $u$ axis and the line $\theta=2\arctan(a)$. Then, we take the odd periodic extension $\overline{\omega}_a$ of $\omega_a^*(\theta,r)$ with respect to $\theta$ about the two asymptotic lines. Now, since the period of $\overline{\omega}_{a}$ is $4\arctan(a)$, to guarantee that the shape is continuous at $\theta=0$ we must have that $a$ satisfies $\arctan(a)=\pi/2n$, where $n$ is the number of waves in the configurations profile (see figure \ref{periodicprofiles}). 

\begin{figure}[htp]
\begin{center}
\includegraphics[width=\textwidth]{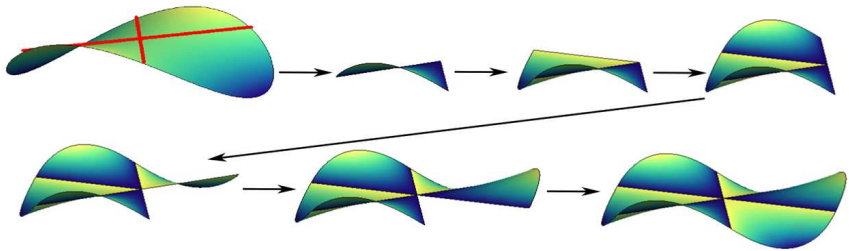}

\caption{ The one parameter family of solutions $\omega_a$ are saddles with asymptotic lines that vanish along the lines through the origin described by the polar angles $\theta=\pm \arctan(a)$ colored in red. Periodic solutions to the Monge-Ampere equation with $n\in \{2,3,\ldots\}$ periods can be constructed by taking odd periodic extensions of the solutions $\displaystyle{\omega_{\tan(\pi/2n)}(u,v)}$ about the line $\theta= n/2\pi$. These solutions are $C^1(D)$ and are smooth everywhere except on the asymptotic lines.} \label{periodicprofiles}
\end{center}
\end{figure}

Since $\omega_a$ is smooth and vanishes along the two lines we are taking the extension about, it is a basic result of Fourier analysis that $\overline{\omega}_a$ has continuous first partial derivatives but a jump in the second partial derivative with respect to $\theta$. Consequently, these lines are singularities in the sense that the surface is not smooth across these lines but the bending energy remains continuous and is finite. Furthermore, since the metric only involves first derivatives of the map $\mathbf{x}$ the stretching energy will still vanish on these lines. Therefore, these surfaces are valid ``isometric" immersions for the small slopes approximation i.e., they are $W^{2,2}$ isometric immersions.  The elastic energy of a disk with radius $R$ and $n$ periods is simply:
\begin{align} \label{linear-energy}
\mathfrak{B}\left(\overline{\omega}_{\tan\left(\frac{\pi}{2n}\right)}\right)&=2n\int_{\frac{-\pi}{2n}}^{\frac{\pi}{2n}}\int_{0}^{R}|D^2\omega_{\tan\left(\frac{\pi}{2n}\right)}|^2r\,dr\,d\theta \nonumber\\
&=\pi R^2\left(\tan^2\left(\frac{\pi}{2n}\right)+\cot^2\left(\frac{\pi}{2n}\right)\right).
\end{align}

For any solution of  (\ref{Monge-Ampere}), the product of the eigenvalues of the Hessian $D^2 \omega$ is $k_1k_2 = -1$. Consequently,
$$
\|D^2 \omega\|^2 = k_1^2 + k_2^2 = k_1^2 + k_1^{-2} \geq 2
$$
with equality only if $k_1^2 = k_2^2 = 1$. With $n = 2$ ({\em i.e.} $a = 1$) in the periodic profiles from above, the Hessian of the surface has eigenvalues $k_{1,2}=\pm 1$. Consequently, the 2-wave saddle $\omega(u,v) = \frac{1}{2}(u^2-v^2)$ gives a global minimum for the bending energy over all the solutions of (\ref{Monge-Ampere}). 

Using this as a test function for the elastic energy in (\ref{reduced-energy}), we get an upper bound for the energy of the minimizer which scales like $t^2$.  It is a standard argument to show that for all $t >0$ there exist minimizers for the reduced elastic energy $E_t$ \cite{evans}. Further, if the appropriate $\Gamma$-limit of $E_t$ exists then these minimizing configurations and their derivatives will converge as $t\rightarrow 0$ to a limit configuration which satisfies satisfies (\ref{Monge-Ampere}) and minimizes the bending energy (\ref{smallslopes-energy}) \cite{braides}. We provide a more detailed argument proving these statements below.

\subsection{$\Gamma$-convergence and convergence of minimizers}
Recall that on a metric space $X$ a sequence of functionals $F_t:X\rightarrow \mathbb{R}$ $\Gamma$-converge as $t\rightarrow 0$ to a functional $F_0:X\rightarrow \mathbb{R}$, written $\Gamma-\lim_{t\rightarrow 0}F_t=F_0$, if for all $x\in X $ we have that

\begin{enumerate}
\item (\textbf{Liminf Inequality}) for every sequence $x_j$ converging to $x$ 
\begin{equation} \label{gamma-lsc}
F_0[x]\leq \liminf_jF_j[x_j];
\end{equation}
\item (\textbf{Recovery Sequence}) there exists a sequence $x_j$ converging to $x$ such that
\begin{equation} \label{recovery}
F_0[x]\geq \limsup_jF_t[x].
\end{equation}
\end{enumerate}

\begin{proposition}
Let $F_t$ be the sequence of functionals on $\mathcal{A}=W^{1,2}(B(R))\times W^{1,2}(B(R))\times W^{2,2}(B(R))$ defined by
\begin{equation}
F_t[\mathbf{x}]=F_t[(f,g,\omega)]=\frac{1}{t^2}E[\mathbf{x}],
\end{equation}
where we are naturally identifying a configuration $\mathbf{x}$ with an element $(f,g,\omega)\in \mathcal{A}$.
Then
\begin{equation}
\Gamma-\lim_{t\rightarrow 0} F_t[(f,g,\omega)]=F_0[(f,g,\omega)]=\begin{cases}
\mathfrak{B}[\omega] &\text{ if } \gamma_{11}=\gamma_{12}=\gamma_{22}=0\\
\infty & \text{otherwise}
\end{cases}
\end{equation}
\end{proposition}
\begin{proof} $\,$
Let $\mathbf{x}_t=(f_t,g_t,\omega_t)\in \mathcal{A}$ such that $\mathbf{x}_t\rightharpoonup \mathbf{x}=(f,g,\omega)$.
\begin{enumerate}
\item It is a basic property of weak convergence that $\forall \eta\in L^2(B(R))$ if $\eta_t\rightharpoonup \eta$ in $L^2(B(R))$ then $\|\eta\|_{L^2}\leq \liminf_{t\rightarrow 0}\|\eta_t\|_{L^2}$ and $\exists C>0$ such that $\|\eta_t\|_{L^2}<C$.\\
\item It follows immediately from item 1 that
\begin{equation*}
\mathfrak{B}[\mathbf{x}]\leq \liminf_{t\rightarrow 0}\mathfrak{B}[\mathbf{x}_t].
\end{equation*}
\item
  Since $B(R)$ has finite area it follows that $\frac{\partial \omega_t}{\partial u}\in W^{1,p}(B(R))$ for $1\leq p \leq 2$ and thus by Rellich's compactness theorem we have that $\frac{\partial \omega_t}{\partial u}$ converges strongly to $\frac{\partial \omega}{\partial u}$ in the $L^q$ norm for $1\leq q<\infty$. Therefore for $1\leq q<\infty$  we have by the reverse triangle inequality that
  \begin{equation*}
  \left|\left\|\frac{\partial \omega_t}{\partial u}\right\|_{L^q}-\left\|\frac{\partial \omega}{\partial u}\right\|_{L^q}\right|\leq \left\|\frac{\partial \omega_t}{\partial u}-\frac{\partial \omega}{\partial u}\right\|_{L^q}
  \end{equation*}
  and thus $\lim_{t\rightarrow 0}\left\|\frac{\partial \omega_t}{\partial u}\right\|_{L^q}=\left\|\frac{\partial \omega}{\partial u}\right\|_{L^q}$. Therefore we have the following limits and inequalities:
  \begin{enumerate}[i]
  \item
  \begin{equation*}
\lim_{t\rightarrow 0}\left\|\frac{\partial f_t}{\partial u}\right\|_{L^2}^2\geq  \|f\|_{L^2}^2,
\end{equation*}
\item
\begin{eqnarray*}
\lim_{t\rightarrow 0}\int_{B(R)}\frac{\partial f_t}{\partial u}\left(\frac{\partial \omega_t}{\partial u}\right)^2\,dA&=&\lim_{t\rightarrow 0}\int_{B(R)}\frac{\partial f_t}{\partial u}\left[\left(\frac{\partial \omega_t}{\partial u}\right)^2-\left(\frac{\partial \omega}{\partial u}\right)^2\right]\,dA\\
&\,& +\lim_{t\rightarrow 0} \int_{B(R)}\frac{\partial f_t}{\partial u}\left(\frac{\partial \omega}{\partial u}\right)^2\,dA\\
&=& \int_{B(R)}\frac{\partial f}{\partial u}\left(\frac{\partial \omega}{\partial u}\right)^2\,dA,
\end{eqnarray*}
\item 
\begin{equation*}
\lim_{t\rightarrow 0}\int_{B(R)}\frac{\partial f_t}{\partial u}\frac{v^2}{3}\,dA=\int_{B(R)}\frac{\partial f}{\partial u}\frac{v^2}{3}\,dA,
\end{equation*}
\item 
\begin{equation*}
\lim_{t\rightarrow 0}\left\|\frac{\partial \omega_t}{\partial u}\right\|_{L^4}^4=\left\|\frac{\partial \omega}{\partial u}\right\|_{L^4}^4,
\end{equation*}
\item 
\begin{equation*}
\lim_{t\rightarrow 0}\int_{B(R)}\left(\frac{\partial \omega_t}{\partial u}\right)^2\frac{v^2}{3}\,dA=\int_{B(R)}\left(\frac{\partial \omega}{\partial u}\right)^2\frac{v^2}{3}\,dA.
\end{equation*}
\end{enumerate}

Therefore by items (i-v) it follows that 
\begin{equation*}
\liminf_{t\rightarrow 0}\int_{B(R)}\left(2\frac{\partial f_t}{\partial u}+\left(\frac{\partial \omega_t}{\partial u}\right)^2-\frac{v^2}{3}\right)^2\,dA\geq \int_{B(R)}\left(2\frac{\partial f}{\partial u}+\left(\frac{\partial \omega}{\partial u}\right)^2-\frac{v^2}{3}\right)^2\,dA.
\end{equation*}
Identical arguments can be used on the other strain terms to show that
\begin{equation*}
\mathcal{S}[\mathbf{x}]\leq\liminf_{t\rightarrow 0}\mathcal{S}[\mathbf{x}_t].
\end{equation*}
\end{enumerate}

If $\mathcal{S}[\mathbf{x}]=0$ then by item 2 we have that
\begin{equation*}
\liminf_{t\rightarrow 0}F_t[\mathbf{x}_t]=\liminf_{t\rightarrow 0}\frac{1}{t^2}\mathcal{S}[\mathbf{x}_t]+\liminf_{t\rightarrow 0}\mathfrak{B}[\mathbf{x}_t]\geq \mathfrak{B}[\mathbf{x}],
\end{equation*}
while if $\mathcal{S}[\mathbf{x}]\neq 0$ then by item 3 
\begin{equation*}
\liminf_{t\rightarrow 0}F_t[\mathbf{x}_t]\geq\liminf_{t\rightarrow 0}\frac{1}{t^2}\mathcal{S}[\mathbf{x}_t]+\liminf_{t\rightarrow 0}\mathfrak{B}[\mathbf{x}_t]\geq \liminf_{t\rightarrow 0}\frac{1}{t^2}\liminf_{t\rightarrow 0}\mathcal{S}[\mathbf{x}_t]=\infty.
\end{equation*}
Therefore, $\liminf_{t\rightarrow 0}F_t[\mathbf{x}_t]\geq F_0[\mathbf{x}]$ proving property 1 of the definition. Finally, if $\mathbf{y}\in \mathcal{A}$ then by selecting the sequence $\mathbf{y}_t=\mathbf{y}$ we prove property 2.
\end{proof}

\begin{lemma}
Let $a_t$ be the sequence in $\mathbb{R}^+$ defined by $a_t=\min_{\mathcal{A}}F_t$. Then, $\lim_{t\rightarrow 0}a_t$ exists and $0\leq\lim_{t\rightarrow 0}a_t\leq 2\pi$.
\end{lemma}
\begin{proof}
Fix $s,t>0$ such that $s<t$ and let $\mathbf{x}_s,\mathbf{x}_t\in \mathcal{A}$ such that $\min_{\mathcal{A}}F_t=F_t[\mathbf{x}_t]$ and $\min_{\mathcal{A}}F_s=F_s[\mathbf{x}_s]$. Then,
\begin{equation*}
F_t[\mathbf{x}_t]\leq F_t[\mathbf{x}_s]=\frac{1}{t^2}S[\mathbf{x}_s]+\mathfrak{B}[\mathbf{x}_s]< \frac{1}{s^2}S[\mathbf{x}_s]+\frak{B}[\mathbf{x}_s]=F_s[\mathbf{x}_s].
\end{equation*}
Therefore, the sequence $a_t=\min_{\mathcal{A}}F_t$ is monotone increasing as $s\rightarrow 0$ and by (\ref{linear-energy}), with $n=2$, satisfies $0\leq a_t\leq 2\pi R$. Therefore by the Bolzano--Weierstrass theorem $a_t$ converges.
\end{proof}

\begin{definition}
Define $\mathcal{A}_0\subset \mathcal{A}$ by
\begin{equation*}
 \mathcal{A}_0=\left\{(f,g,\omega)\in \mathcal{A}: f(0)=g(0)=\omega(0)=0 \text{ and } \left.\frac{\partial \omega}{\partial u}\right|_{(0,0)}=\left.\frac{\partial \omega}{\partial v}\right|_{(0,0)}\right\}=0.
 \end{equation*}
\end{definition}

\begin{remark}
The above definition fixes the origin of a configuration and aligns the normal at the origin with the $z$-axis of $\mathbb{R}^3$. By a rigid translation and rotation any element $x\in\mathcal{A}$ can naturally be identified with an element of $\mathcal{A}_0$.
\end{remark}

\begin{theorem}
If $\mathbf{x}_t\in \mathcal{A}_0$ is a sequence such that $\min_{\mathcal{A}_0}F_t=F_t[\mathbf{x}_t]$ then
\begin{equation}
2\pi R=\min_{\mathcal{A}_0}=\lim_{t\rightarrow 0}F_t[\mathbf{x}_t]
\end{equation}
and every limit of convergent subsequence of $\mathbf{x}_t$ is a minimum point for $F_0$.
\end{theorem}
\begin{proof}
Let $\mathbf{x}=(f,g,\omega)\in \mathcal{A}_0$ such that for all $t\in(0,1)$, $F_t[\mathbf{x}]\leq 2\pi R$. Then there exists $C_1>0$ such that $\|\nabla f\|_{L^2},\|\nabla g\|_{L^2}\|\nabla \omega\|_{L^4}<C_1$. Therefore by Poincar\'{e}'s inequality there exists $C_2$ such that $\|f\|_{L^2},\|g\|_{L^2},\|\omega\|_{L^2}<C_2$ and thus by the Banach-Alaoglu theorem the set $K=\{\mathbf{x}\in \mathcal{A}_0:\forall t\in(0,1), F_t[\mathbf{x}]\leq 2\pi R\}$ is weakly compact.
\begin{enumerate}
\item Let $\mathbf{x}_t$ be a sequence such that $\min_{\mathcal{A}_0}F_t=F_t[\mathbf{x}_t]$. By (\ref{linear-energy}) it follows that $F_t[\mathbf{x}_t]\leq 2\pi R$ and thus by compactness of $K$ there exists a subsequence $\mathbf{x}_{t_k}$ such that $\mathbf{x}_{t_k}\rightharpoonup \mathbf{x}^*\in \mathcal{A}_0$. Therefore, by (\ref{gamma-lsc}) and the above lemma we have that 
\begin{equation*}
\min_{\mathcal{A}_0}F_0\leq F_0[\mathbf{x}^*]\leq \liminf_{k\rightarrow 0} F_k[\mathbf{x}_{t_k}]=\lim_{k\rightarrow 0}\min_{\mathcal{A}_0}F_{t_k}=\lim_{t\rightarrow 0} \min_{\mathcal{A}_0}F_t.
\end{equation*} 
\item Fix $\delta >0$ and let $\mathbf{y}\in \mathcal{A}$ such that $F_0[\mathbf{y}]\leq \min_{\mathcal{A}_0}F_0+\delta$. Then if $\mathbf{y}_j$ is a sequence satisfying (\ref{recovery}) then 
\begin{equation*}
\min_{\mathcal{A}_0}F_0+\delta \geq F_0[\mathbf{y}]\geq \limsup_{t\rightarrow 0}F_t[\mathbf{y_t}]\geq \limsup_{t\rightarrow 0}\min_{\mathcal{A}_0}F_t=\lim_{t\rightarrow 0}\min_{\mathcal{A}_0}F_t.
\end{equation*}
By the arbitrariness of $\delta$ if follows that 
\begin{equation*}
\min_{\mathcal{A}_0}F_0\geq \lim_{t\rightarrow 0}\min_{\mathcal{A}_0}F_t.
\end{equation*}
\end{enumerate}
Therefore, by items 1 and 2 it follows that
\begin{equation*}
\min_{\mathcal{A}_0}F_0=\lim_{t\rightarrow 0}\min_{\mathcal{A}_0}F_t.
\end{equation*}
Furthermore, if $\mathbf{x}_t$ is sequence as defined in item 1 and $\mathbf{x}_{t_k}$ is a subsequence such that $\mathbf{x}_{t_k}\rightharpoonup \mathbf{x}^*$ then 
\begin{equation*}
\lim_{t\rightarrow 0}\min_{\mathcal{A}_0}F_t=\min_{\mathcal{A}_0}F_0\leq F_0[\mathbf{x}^*]\leq \lim_{t\rightarrow 0} \min_{\mathcal{A}_0}F_t
\end{equation*}
which proves that
\begin{equation*}
\min_{\mathcal{A}_0}=F_0[\mathbf{x}^*].
\end{equation*}
\end{proof}

 Consequently, the small slopes approximation will predict that, for small thickness, the configuration of the sheet will converge to a quadratic saddle $\omega = \frac{1}{2}(u^2-v^2)$. This is in disagreement with the experimental observation that the disks can obtain an arbitrary number of waves \cite{an-experimental-study-of-shape-transitions-and-energy-scaling-in-thin-non-euclidean-plates}. However, if we consider the amplitude $A_n$ of an $n$ wave configuration, we get the scaling law
\begin{equation}
A_n(R)=\frac{1}{2}\tan\left(\frac{\pi}{2n}\right)R^2\approx \frac{\pi}{4}\frac{R^2}{n},
\end{equation}
which does agree with the experiment \cite{an-experimental-study-of-shape-transitions-and-energy-scaling-in-thin-non-euclidean-plates}. 

A key difference between exact isometries and the small slopes approximation is captured by the following observation. The curvatures of the surfaces given by (\ref{quad-surface}) are bounded by $\max(a,a^{-1}) \sim O(1)$ independent of $R$, the radius of the disk. This is in contrast with our result (\ref{eq:sclng}) which shows that for isometric immersions, the curvature has to grow exponentially in the radius $R$. A qualitative expression of this distinction is that the analogs of Hilbert's and Efimov's theorems do not hold for the small slopes approximation. Consequently, in the next section we look to extend our construction by looking for periodic surfaces that are exactly locally isometric to $\mathbb{H}^2$ and for small  $R$ are well approximated by the small slopes solution.

%%%%%%%%%%%%%%%%%%%%%%%%%%%%%%%%%%%%%%%%%%%%%%%%%%%%%%%%%%
%
% Amsler Surfaces
%
%%%%%%%%%%%%%%%%%%%%%%%%%%%%%%%%%%%%%%%%%%%%%%%%%%%%%%%%%%%%%%%%
\section{Periodic Amsler Surfaces} \label{sec:amsler}

To mimic the solutions to the small slope approximation we want to construct a hyperbolic surface that has two asymptotic curves that are straight lines and intersect at the origin of $U$. Hyperbolic surfaces satisfying this property exist, they are called Amsler surfaces \cite{Bobenko}, and form a one parameter family of surfaces $\mathcal{A}_{\theta}$ that are uniquely determined by the angle $\theta$ between the asymptotic lines \cite{Amsler}. As in the construction of surfaces in the small slopes approximation, if the angle $\theta$ between the asymptotic lines satisfies $\theta=\frac{\pi}{n}$ we can take the odd periodic extension of the piece of the Amsler surface bounded between the asymptotic lines and form a periodic profile with $n$ waves. We call the surfaces constructed in this manner periodic Amsler surfaces.

The key to generating periodic Amsler surfaces is through the Sine-Gordon equation (\ref{Sine-Gordon}) and the similarity transformation $z=2\sqrt{uv},\, \varphi(z)=\phi(u,v)$ which transforms (\ref{Sine-Gordon}) into a Painlev\'e III equation in trigonometric form:
\begin{equation} \label{Painleve-Equation}
\varphi^{\prime \prime}(z)+\frac{1}{z}\varphi^{\prime}(z)-\sin(\varphi(z))=0,
\end{equation}
where $\prime$ denotes differentiation with respect to $z$ \cite{Amsler, Bobenko}. By imposing the initial conditions $\varphi(0)=\frac{\pi}{n}$ and $\varphi^{\prime}(0)=0$ solutions to (\ref{Painleve-Equation}), denoted by $\varphi_n(z)$, generate surfaces such that the $u$ and $v$ axis correspond to the asymptotic lines of the surface. Therefore, the piece of the surface we are extending periodically is parameterized in the first quadrant of the $u\--v$ plane and all of its geometric quantities are determined by $\varphi_n$. 

\begin{figure}[htp]
\begin{center}
\includegraphics[width=\textwidth]{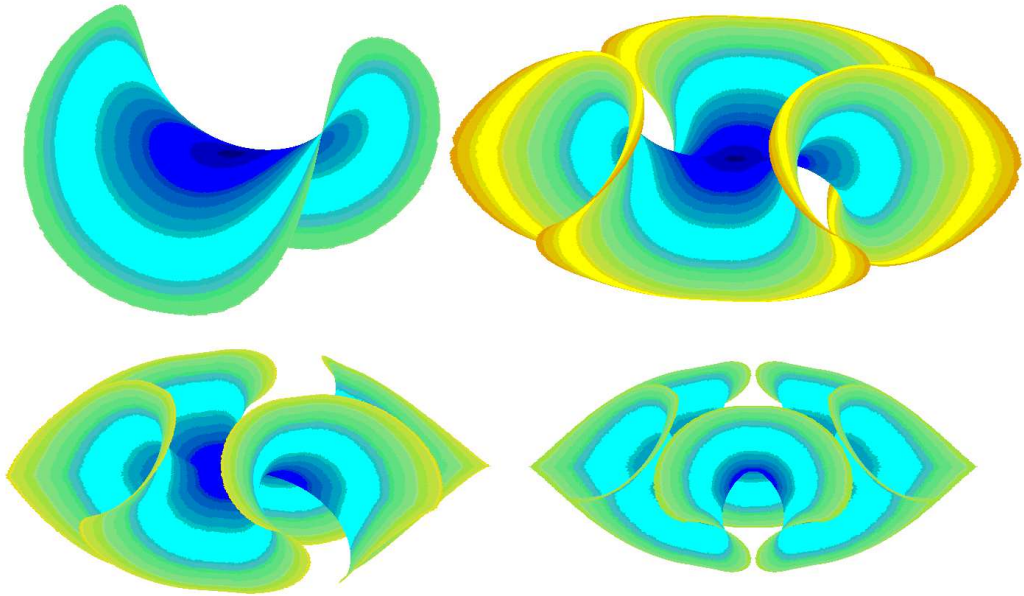}
\caption{Periodic Amsler surfaces with $n=2,3,4$ and $5$ waves respectively. These surfaces are not drawn to scale, but are close to the maximum radii presented in figure \ref{amslerenergy:bendenergy(r)}. The coloring of the disks corresponds to contours of arclength data and indicate how the different geodesic circles cut from these surfaces would appear. These surfaces are in fact not true Amsler surfaces but are discrete Amsler surfaces creating using the algorithm presented in \cite{Bobenko-1999}. }\label{amslersurfaces} 
\end{center}
\end{figure}  

\subsection{Maximum radius of periodic Amsler surfaces}
The geodesic equations (\ref{geodesic-equations}) under this similarity transformation take the form: 
\begin{equation} \label{geodesics}
\begin{array}{c}
\frac{d^2u}{dt^2}+\frac{1}{\sqrt{uv}}\frac{d\varphi_n}{dz}\left(\cot(\varphi_n)v\left(\frac{du}{dt}\right)^2-\csc(\varphi_n)u\left(\frac{dv}{dt}\right)^2\right)=0,\\
\frac{d^2v}{dt^2}+\frac{1}{\sqrt{uv}}\frac{d\varphi_n}{dz}\left(\cot(\varphi_n)u\left(\frac{dv}{dt}\right)^2-\csc(\varphi_n)v\left(\frac{du}{dt}\right)^2\right)=0,
\end{array}
\end{equation}
with initial conditions $u(0)=v(0)=0$, $\frac{du}{dt}(0)=\cos(\psi),$ and $\frac{dv}{dt}(0)=\sin(\psi)$ for $\psi\in[0,\pi/2]$. Notice that these equations have a singularity at the point $z_n=\varphi_n^{-1}(\pi)$. This is precisely where the immersion becomes singular and thus the singular curve is given in the $u\--v$ plane by 
\begin{equation}
\sqrt{v}=\frac{\varphi_n^{-1}(\pi)}{2\sqrt{u}}.
\end{equation} 
Once the singular curve is determined we can calculate the maximum radius of a periodic Amsler surface with $n$ waves by determining the shortest geodesic that intersects the singular curve.

We can obtain a scaling for $\varphi_n^{-1}(\pi)$ by multiplying equation (\ref{Painleve-Equation}) by $\varphi^{\prime}(z)$ to get 
\begin{equation}
\frac{1}{2}\frac{d(\varphi^{\prime})^2}{d z}=-\frac{d\varphi}{dz}\frac{d\cos \varphi}{d \varphi}-\frac{1}{z}\varphi^{\prime}(z)^2\leq -\frac{d \cos(\varphi)}{d z}.
\end{equation}
Integrating, we get that
\begin{equation}
\int_{\pi/n}^{\pi}\frac{d \varphi}{\sqrt{2(1-\cos(\varphi))}}\leq \int_{\pi/n}^{\pi} \frac{d \varphi}{\sqrt{2(\cos(\pi/n)-\cos(\varphi)}}\leq \varphi_n^{-1}(\pi).
\end{equation}
Therefore,
\begin{equation}\label{pi-scaling}
\ln\left(\cot\left(\frac{\pi}{4n}\right)\right)\leq \varphi_n^{-1}(\pi).
\end{equation}

 We can use this scaling to prove that we can create arbitrary large disks from these periodic Amsler surfaces.  First noting that the same argument used to obtain (\ref{pi-scaling}) implies that $\displaystyle{\lim_{n \rightarrow \infty}\varphi_{n}^{-1}(\pi/2)=\infty}$. Now, let $\alpha(t)=\mathbf{x}((u(t),v(t)))$ be a geodesic lying on a periodic Amsler surface with generating angle $\theta=\pi/n$. Suppose $\alpha(t)$ starts at the origin and travels to the curve $\displaystyle{\sqrt{v}=\frac{\varphi_n^{-1}(\pi/2)}{2\sqrt{u}}}$ terminating at the point $(u_f,v_f)$ and without loss of generality assume that $u(t)$ is a function of $v(t)$. Then, the arclength of $\alpha(t)$ satisfies:
\begin{equation} \label{Amsler-Radius-Scaling}
d(0,(u_f,v_f))=\int_0^{u_f}\sqrt{1+2\cos(\varphi)\frac{du}{dv}+\left(\frac{du}{dv}\right)^2}\,dv>\int_0^{u_f}\,dv=u_f.
\end{equation} 
Thus, since $u_f\rightarrow \infty$ as $n\rightarrow \infty$ we have proved the following proposition.
\begin{proposition}
Let $D$ be a disk of radius $R$ in the hyperbolic plane. There exists a $W^{2,2}$ isometric immersion $\mathbf{x}:D\rightarrow U\subset \mathbb{R}^3$ such that $U$ is a subset of a periodic Amsler surface.
\end{proposition}
 
Figure \ref{scaling} is a plot of several solutions to equation (\ref{Painleve-Equation}) and includes the scaling of the radius of the largest disk with $n$ waves that can be embedded in $\mathbb{R}^3$. We can see qualitatively that the maximum radius grows logarithmically with $n$, a result that is confirmed by equation (\ref{pi-scaling}). The geodesics for the particular case when $n=2$  are plotted in figure \ref{amsler-geodesics:geodesics} along with the geodesic circles for $n=3$ in figure \ref{Amsler-Geodesics:circles}.  

\begin{figure}[htp]
\begin{center}
\includegraphics[width=.8\textwidth]{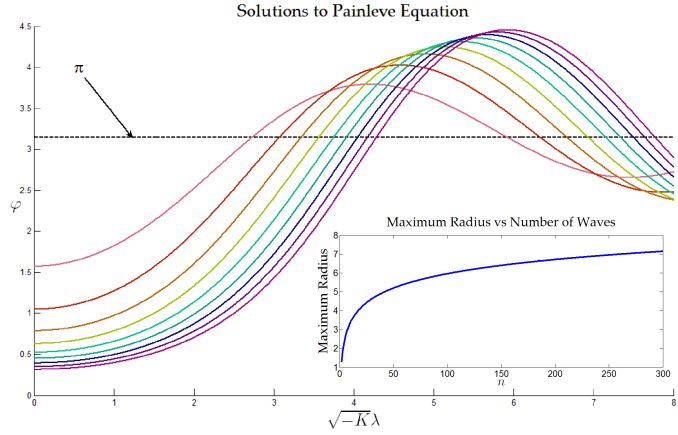}
\caption{A plot of $\varphi$ versus $r$ satisfying the initial conditions $\varphi(0)=\pi/n$ for $n=2, \ldots 10$. The dashed horizontal line corresponds to where $\varphi(r)=\pi$. At this point a principal curvature diverges and by choosing a higher value of $n$ a larger disk can be created. The inset plot illustrates how the maximum radius scales with $n$. We can see that the radius grows very slowly with $n$ and it looks to be growing approximately at a logarithmic rate.}  \label{scaling}
\end{center}
\end{figure}

\begin{figure}[htp]
\begin{center}
\subfigure[]{
\includegraphics[width=.4\textwidth]{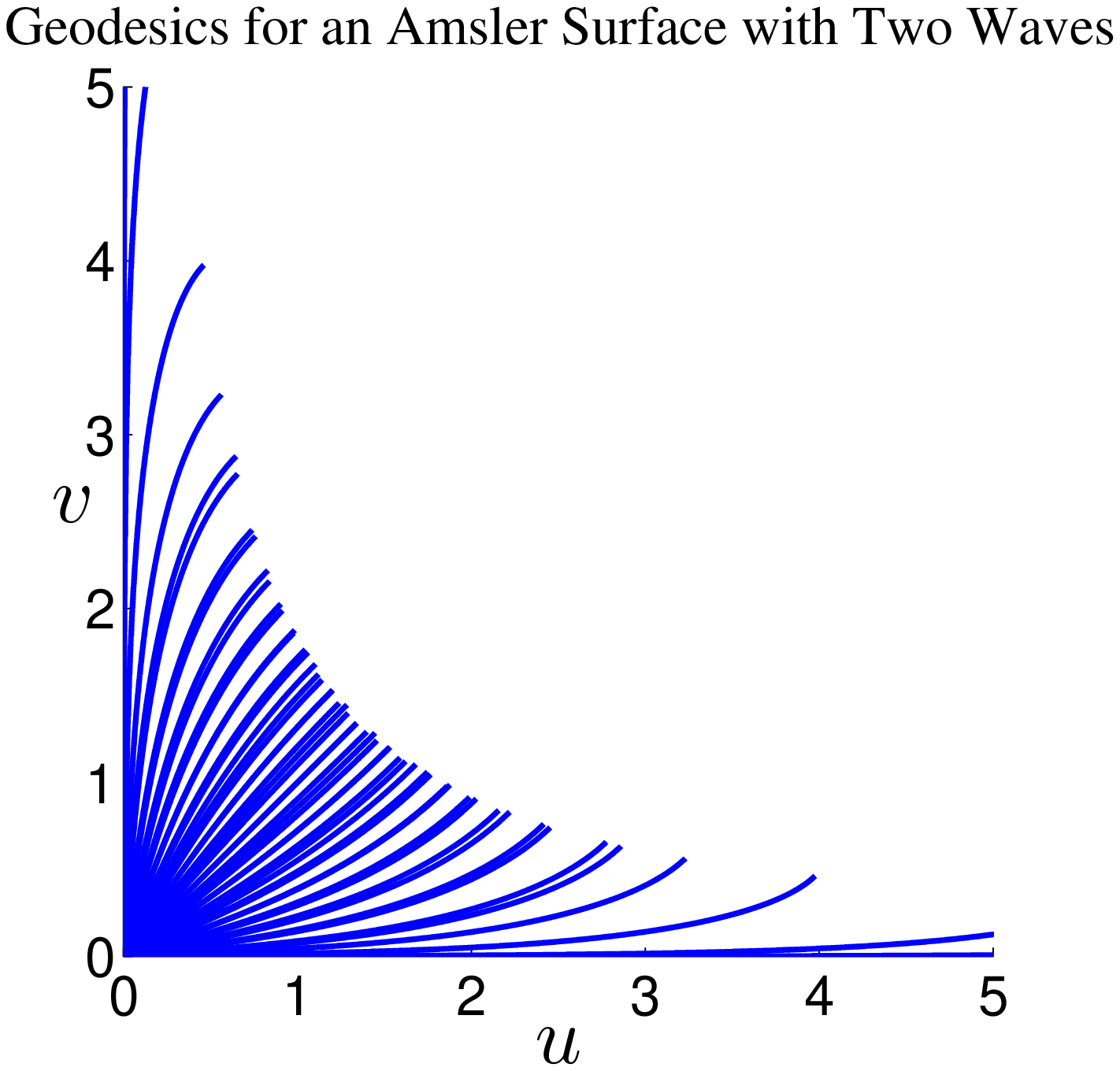}
\label{amsler-geodesics:geodesics}
}
\subfigure[]{
\includegraphics[width=.4\textwidth]{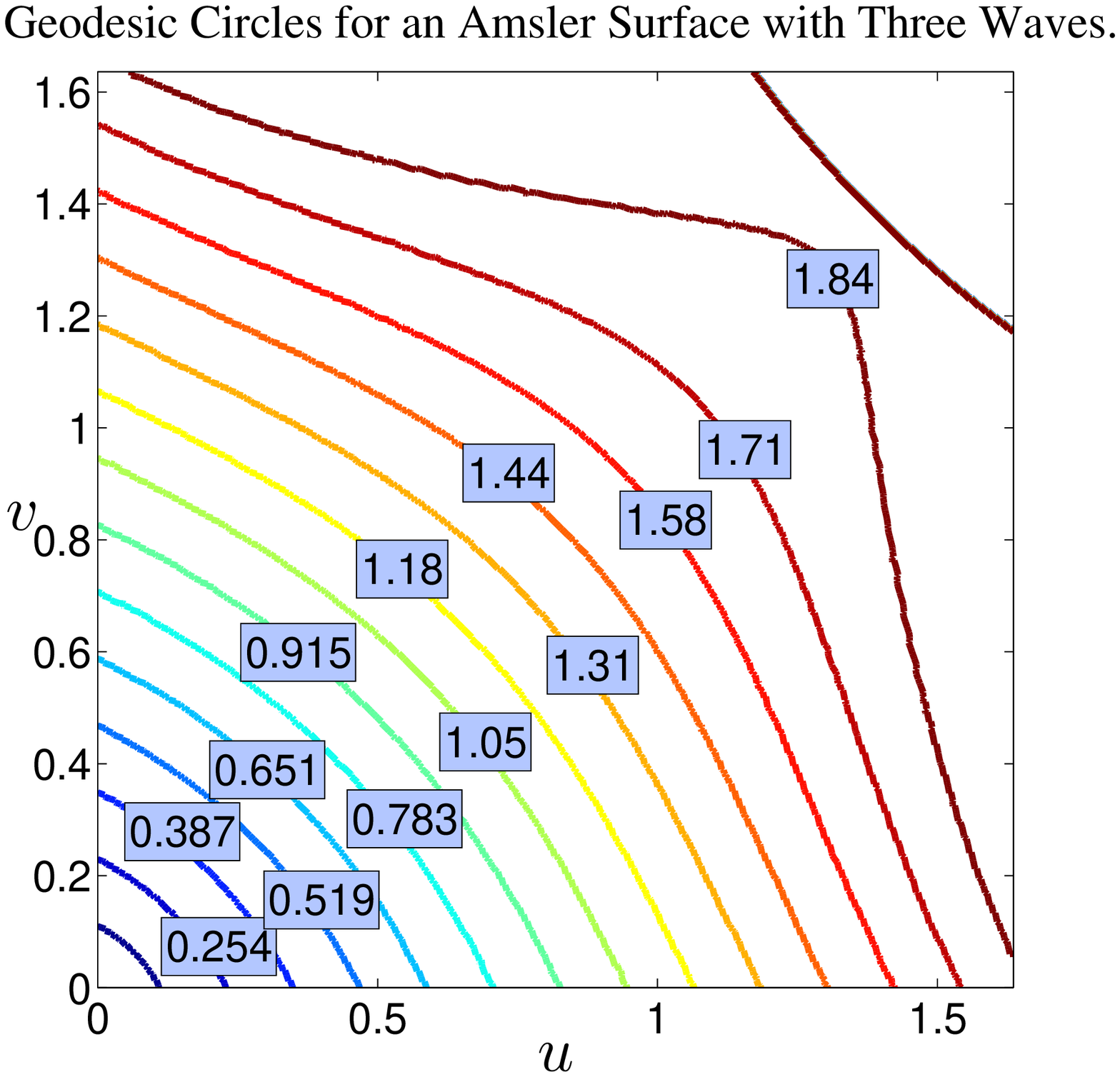}
\label{Amsler-Geodesics:circles}
}
\caption{\subref{amsler-geodesics:geodesics} Plot of the geodesics in the $u\--v$ plane of the Amsler surface for $n=2$ waves. The endpoints of the geodesics lie on the curve $\sqrt{v}=\frac{\varphi_n^{-1}(\pi)}{2\sqrt{u}}$, which forms the boundary where a principal curvature diverges. \subref{Amsler-Geodesics:circles} Geodesic circles plotted in the $u\--v$ plane when $n=3$. These curves where calculated by contouring the arclength of geodesics. The curve in the upper right hand corner is the singular edge of the Amsler surface.} \label{Amsler-Geodesics}
\end{center}
\end{figure}

\subsection{Elastic energy of periodic Amsler surfaces}
To compute the elastic energy using equation (\ref{bending-energy}) we need to determine the dependence of $u$ and $v$ on $r$ and $\Psi$. Let $j_1$ and $j_2$ denote unit vectors in the $u\--v$ plane that are aligned with the $u$ and $v$ axis respectively. If we let $\mathbf{x}$ denote the parameterization of the Amsler surface, then the pushforward or differential of $\mathbf{x}$, denoted $\mathbf{x}_*$, is a linear map. Furthermore, since $\mathbf{x}$ is an isometry and the images of $j_1$ and $j_2$ under $\mathbf{x}_*$ make an angle of $\pi/n$ we have without loss of generality that
\begin{equation}
\mathbf{x}_*j_1=e_1 \text{ and } \mathbf{x}_*j_2=\cos(\pi/n)e_1+\sin(\pi/n)e_2,
\end{equation}
where $e_1$ and $e_2$ are the standard basis elements for the tangent plane at $\mathbf{x}(0,0)$. Consequently, if we let $\alpha(t)=(u(t),v(t))$ be a geodesic defined in the $u\--v$ plane satisfying $\alpha^{\prime}(0)=\cos(\psi)j_1+\sin(\psi)j_2$  then by linearity we have that
\begin{equation}
\mathbf{x}_*(\alpha^{\prime}(0))=\left(\cos(\theta)+\sin(\theta)\cos(\pi/n)\right)e_1+\sin(\theta)\sin(\pi/n)e_2.
\end{equation} 
Therefore, we can conclude that the polar angle $\Psi$ for this geodesic is given by
\begin{equation}\label{Polar-Angle}
\Psi=\arctan\left(\frac{\sin(\psi)\sin(\pi/n)}{\cos(\psi)+\sin(\psi)\cos(\pi/n)}\right).
\end{equation}

Now, we can numerically integrate the elastic bending energy for the periodic Amsler surfaces. First, by specifying $\Psi$ and numerically solving (\ref{Polar-Angle}) for $\psi$ we can generate initial conditions for equations (\ref{geodesics}) which can then be numerically solved to determine $u(t)$ and $v(t)$.  Then, by fixing $r$ and numerically solving the arclength equation
\begin{equation}
r=\int_0^T\sqrt{\left(\frac{du}{dt}\right)^2+2\cos(\varphi)\frac{du}{dt}\frac{dv}{dt}+\left(\frac{dv}{dt}\right)^2}\,dt,
\end{equation}   
for $T$ we can calculate the values $u(T)$ and $v(T)$ which correspond to the coordinates $(\Psi,r)$. Finally by setting up a mesh on the rectangle $(\Psi,r)\in [0,2\pi]\times [0,R]$ and using the above process to determine $u$ and $v$ we can numerically integrate equation (\ref{bending-energy}). These results are summarized in figure \ref{amslerenergy}. 

\begin{figure}[htp] 
\begin{center}

\subfigure[]{
\includegraphics[width=.7\textwidth]{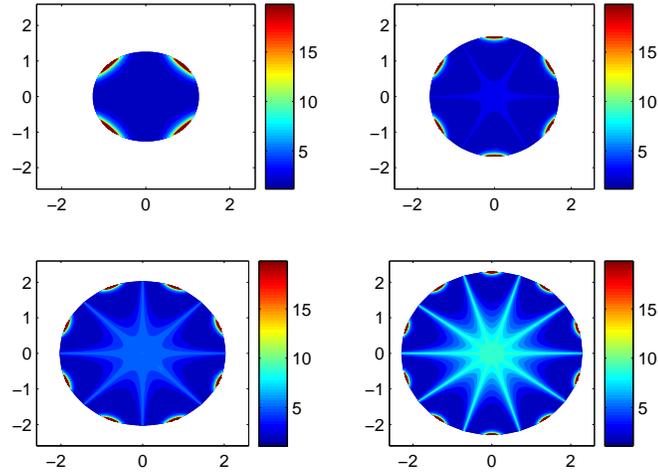}
\label{amslerenergy:discs}
}

\subfigure[]{
\includegraphics[width=.7\textwidth]{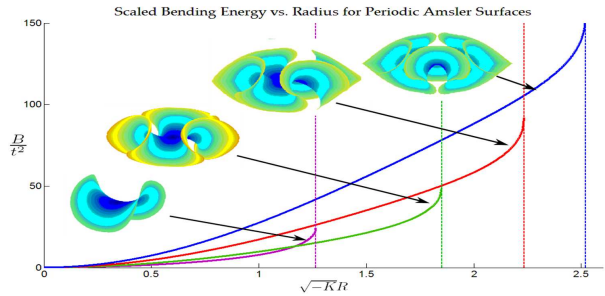} 
\label{amslerenergy:bendenergy(r)}
}

\end{center}
\caption{\subref{amslerenergy:discs}  A representation of the geodesic disks cut out of Amsler surfaces with $n=2,\ldots, 5$ waves colored by $k_1^2+k_2^2$ in which the radius and polar angle correspond to the geodesic radius $r$ and polar angle $\Psi$. These disks each have a maximum radii of $R=1.2654,\, 1.8505,\, 2.2342,\, 2.5199$ beyond which one of the principal curvatures diverges. We can see that these disks have lower energy than there counterparts lying in the hyperboloid of revolution. Moreover, the energy is concentrated in small regions near the singular edge of the disk. \subref{amslerenergy:bendenergy(r)} The scaled bending energy of periodic Amsler surfaces with $n=2,\ldots, 5$ waves plotted versus geodesic radius. The vertical dashed lines correspond to the radius where the surface with $n$ waves can no longer be isometrically embedded.}
\label{amslerenergy}
\end{figure}

%%%%%%%%%%%%%%%%%%%%%%%%%%%%%%%%%%%%%%%%%%%%%%%%%%%%%%%%%%%%%%%%%%%%%%%%%%%%%%%
%
%		Discussion
%
%%%%%%%%%%%%%%%%%%%%%%%%%%%%%%%%%%%%%%%%%%%%%%%%%%%%%%%%%%%%%%%%%%%%%%%%%%%%%%%
\section{Discussion}
\label{sec:discuss}

Free non-Euclidean thin elastic sheets arise in a variety of physical \cite{Sharon2002, Shaping-of-elastic-sheets-by-prescription-of-non-euclidean-metrics,Cerda-2003} and biological \cite{UtpalNath02282003, Geometry-And-Elasticity-of-srips-and-flowers, Dervaux(2008)}  systems. The morphology of these sheets is usually modeled as the equilibrium configurations for an appropriate elastic energy. One approach is to model the sheet as an abstract manifold with a prescribed {\em target metric} $\mathbf{g}$, which is then used to define strains and hence an elastic energy for a configuration. 

In this paper we considered the geometric problem of immersing disks with {\em constant negative curvature} into $\mathbb{R}^3$. There are two settings for this problem which are relevant. One is the problem of ``exact" isometric immersions 
\begin{equation}
D\mathbf{x}^T D\mathbf{x} = \mathbf{g}
\label{eq:exact}
\end{equation}
 where $\mathbf{x}:D \rightarrow \mathbb{R}^3$ is the configuration of the center surface.  The other problem is a small-slopes approximation with an in and out of plane decomposition, $\mathbf{x} = i + \epsilon \omega + \epsilon^2 \chi, \mathbf{g} = \mathbf{g}_0 + \epsilon^2 \mathbf{g}_1$ with the isometric immersions given by
\begin{equation}
D \chi^T + D\chi + D\omega^T D\omega = \mathbf{g}_1.
\label{eq:lnrzd}
\end{equation}

In the small slopes approximation, we proved that global minimizers of the bending energy over solutions of (\ref{eq:lnrzd}) are given by a quadratic saddle $\omega(u,v) = \frac{1}{2}(u^2-v^2)$. We thus expect that, with decreasing thickness, the configuration of the sheet will converge to a two wave solution. This is not what is observed experimentally, where the number of waves increases with decreasing thickness \cite{an-experimental-study-of-shape-transitions-and-energy-scaling-in-thin-non-euclidean-plates}. Consequently, the physically realized configurations {\em are not} given by the global minimizers of the model energy. Note however, that the periodic, non-smooth profiles that we construct in Sec.~\ref{sec:small_slopes} do agree qualitatively, and in scaling behavior, with the experimental observations. This is a puzzle, since these periodic, non-smooth profiles {\em are not} global minimizers of the energy, and we need a different selection mechanism if these profiles do indeed represent the experimental configurations.

One way to improve the modeling of the system is to consider the ``full" F\"{o}ppl--von K\'{a}rm\'{a}n energy (\ref{energy}) and the associated isometric immersions given by (\ref{eq:exact}). For complete surfaces of constant negative curvature (or surfaces whose curvature is bounded above by a negative constant), Hilbert's theorem \cite{Hilbert} (resp. Efimov's theorem \cite{Efimov-1964}) show that there are no analytic (resp. $C^2$) solutions of (\ref{eq:exact}). These results are sometimes used as a basis for physical intuition that there are no ``global" isometric immersions of constant negative curvature surfaces, and this leads to the observed refinement of wavelength with decreasing thickness \cite{0295-5075-86-3-34003}. 

We investigated this problem and showed (Sec.~\ref{sec:pseudosphere} and ~\ref{sec:minimax}) there exists smooth isometric immersions $\mathbf{x}:D_R \rightarrow \mathbb{R}^3$ for geodesic disks of arbitrarily large radius $R$. Moreover, in elucidating the connection between this existence result and the non-existence results above, we provided numerical evidence that the maximum principal curvature of such immersions is bounded below by a bound which grows exponentially in $\sqrt{-K_0}R$. 

The existence of smooth, and thus also $W^{2,2}$ immersions of arbitrarily large disks have consequences for the modeling of free non-Euclidean sheets. In particular, the minimizers of the elastic energy (\ref{energy}) must converge as $t \rightarrow 0$ to a minimizer of (\ref{bending-energy-general}) \cite{ERAN,Gamma-limit}, and the energy of the minimizer $\mathbf{x}^*$ is bounded from above by
$$
\mathcal{E}[\mathbf{x}^*] \leq t^2 \inf_{\mathbf{y}\in \mathcal{A}}\mathcal{B}[\mathbf{y}]
$$ 
where $\mathcal{A}$ is the set of all the $W^{2,2}$ solutions of (\ref{eq:exact}). This also implies that the bending energy of the minimizing configurations should satisfy the inequality
$$
\mathcal{B}[\mathbf{x}^*] \leq \mathcal{B}_0 \equiv \inf_{\mathbf{y}\in \mathcal{A}}\mathcal{B}[\mathbf{y}] 
$$
for all $t > 0$, {\em independent of} $t$.  This is an apparent contradiction with experimental observations where the bending energy of the minimizer scales with the thickness $t$ and  is well fit by a power law, $t^{-1}$, which diverges as $t\rightarrow 0$ \cite{an-experimental-study-of-shape-transitions-and-energy-scaling-in-thin-non-euclidean-plates}. This suggests that further improvements of the modeling might be necessary for a quantitative description of the experimental phenomena. 

The smooth immersions $\mathbf{x}:D_R \rightarrow \mathbb{R}^3$ (Secs.~\ref{sec:pseudosphere}~and~\ref{sec:minimax}) do not have the rotational $n$-fold symmetry ($\Psi \rightarrow \Psi + 2 \pi/n$ for the geodesic polar angle $\Psi$) of the underlying energy functional, which is also seen in the experimentally observed configurations. In Sec.~\ref{sec:amsler} We  generalize the construction of the non-smooth, $n$-wave small-slopes solutions to obtain isometric immersions (solutions of (\ref{eq:exact})) which are also non-smooth and have the same symmetry/morphology as the experimentally observed configurations. These periodic Amsler surfaces  {\em have lower bending energy} than the smooth immersions given by subsets of the pseudosphere and hyperboloids with constant negative curvature. 

For each $n \geq 2$, there is a radius $R_n \sim \log(n)$ such that the $n$-periodic Amsler surfaces only exist for a radius $0 < R < R_n$. This gives a natural mechanism for the refinement of the wavelength of the buckling pattern with increasing radius of the disk. However, it does not, at least directly, explain the observed refinement with decreasing thickness. 

We have considered the purely geometric model which arises as a limit $t \rightarrow 0$. A natural question is the applicability if our results for physical sheets where $t >0$. Perhaps the refinement of the wavelength with decreasing $t$  is an effect which can only be captured by studying the full F\"{o}ppl--von K\'{a}rm\'{a}n energy (\ref{energy}). For $t >0$,  we expect that the minimizers of the energy  will have boundary layers near the edge of the disk and also near the singular asymptotic lines, so that the disks satisfy stress and moment balance everywhere. The boundary layers near the outer edge of the disk may then set the wavelength throughout the disk. A similar phenomenon in stretched elastic sheets was studied in this manner in \cite{PhysRevLett.91.086105}. 

The buckling patterns in non-Euclidean sheets are the result of dynamical (albeit slow) processes, and  it may be more appropriate to model this type of differential growth dynamically, perhaps as a gradient flow for the elastic energy. The pattern is then selected for dynamical reasons (fastest growth rate, etc.) and not by global minimization of an energy functional. This might explain why local but not global extrema for the energy functional seem to describe the observed patterns. 

However, neither the boundary layers, nor the gradient flow dynamics,  can explain how the refinement with decreasing thickness can lead to an apparent divergence of the the bending energy $\mathcal{B}[\mathbf{x}]  \sim t^{-1}$.  We have shown rigorously, for the elastic energy (\ref{energy}), in the limit $t \rightarrow 0$, we should have $\mathcal{B}[\mathbf{x}] \rightarrow \mathcal{B}_0 < \infty$.  This suggests the need for better models that are carefully validated  by comparison with experiments in order to obtain a quantitative description of non-Euclidean free sheets. 

We conclude our discussion with a short list of mathematical questions and avenues for future research that come out of this work --
\begin{enumerate}
\item In Sec.~\ref{sec:minimax}, we use an {\em ansatz} motivated by numerical experiments to show that the solutions to the minimax problem
$$
\phi_\infty = \mbox{arg}\min_{\phi \in \mathcal{A}} \max_{x \in D} \cot^2(\phi) 
$$
are given by disks cut out of constant curvature hyperboloids of revolution, where the admissible set $\mathcal{A}$ consists of all $C^2$ or smoother isometric embeddings. We are working on a rigorous (ansatz free)  proof of this result.
\item Our lower bound for the maximum principal curvature (and hence the maximum bending energy density) in Sec.~\ref{sec:minimax} can be interpreted as a lower bound on the $L^{\infty}$ norm of the bending energy density $k_1^2 + k_2^2$ over the admissible set, which consists of isometric immersions. The limiting energy and configuration of the non-Euclidean sheet as $t \rightarrow 0$ is given by the lower bound for the $L^1$ norm of the bending energy density.  Motivated by this observation, we are investigating lower bounds for the $L^p$ norms $1 \leq p < \infty$ to connect the known case $p = \infty$ with  the case of interest $p=1$.
\item Chebyshev nets give a natural discretization of  surfaces with constant negative curvature that respects both the intrinsic and the extrinsic geometry of the surface. This idea has already been used for constructing discrete isometric immersions ($K$-surfaces) for such surfaces \cite{Bobenko-1999}. This can be interpreted as a $t \rightarrow 0$ limit for the variational problem (\ref{energy}).  We are  developing numerical methods for the variational problem given by (\ref{energy}) with $t > 0$ that exploits the Chebyshev net structure. 
\item Our results strongly use the fact that the target metric on the surface has constant negative curvature, which then allows us to naturally associate a Chebyshev net with each isometric immersion through an asymptotic parameterization \cite{gray}. However, we believe that our results, in particular the lower bound for the maximum principal curvature (\ref{eq:sclng}), and the construction of $C^{1,1}$ and piecewise smooth isometric immersions can also be extended to other target metrics whose curvature is bounded above by a negative constant. It is natural to study these questions, in particular the difference between smooth and $C^{1,1}$ immersions  using the ideas in \cite{cmpnstd-cmpctnss}.
\item For the periodic Amsler surfaces in Sec.~\ref{sec:amsler}, the global structure, {\em i.e}. the number $n$ of waves is determined by the local structure at the origin. The origin is a {\em bifurcation point} \cite{cheb-bifur} for the Chebyshev net induced by the embedding in the sense that, every point in the complement of the origin has precisely two asymptotic directions, while the origin has $n$ asymptotic directions with $n > 2$. The authors of \cite{cheb-bifur} remark that for surfaces whose total positive curvature is bounded by $2 \pi$, it is always possible to find a global Chebyshev net, except it can have multiple bifurcation points. This suggests the following natural questions:  Are there non-smooth immersions ($C^{1,1}$ or even $W^{2,2}$ immersions) with multiple bifurcation points? 

The geometry of such surfaces will then serve as a model for the observed morphology in many non-Euclidean sheets including torn plastic \cite{Sharon2002} and lettuce leaves \cite{Theory-of-Edges-of-Leaves} which do not have a globally defined number of waves, but rather have local buckling behavior which increases the number of waves as we approach the boundary.
\end{enumerate}

\noindent \textbf{Acknowledgements:} 
The authors wish to thank Eran Sharon, Efi Efrati and Yael Klein for many useful discussions, and for sharing pre-publication experimental data with us. We would also to thank Marta Lewicka and Reza Pakzad for fruitful discussions. Finally we would like to thank the anonymous reviewer for many useful comments that improved the quality of this paper. This work was supported by the US-Israel BSF grant 2008432 and NSF grant DMS--0807501.

\def\cprime{$'$}

\end{document}